\documentclass[10pt]{article}
\usepackage{mathrsfs}

\usepackage{amsmath}
\usepackage{amsxtra}
\usepackage{amscd}
\usepackage{amsthm}
\usepackage{amsfonts}
\usepackage{amssymb}

\usepackage[all]{xy}
\usepackage{bbm}
\usepackage{hyperref}
\usepackage[mathscr]{eucal}

\numberwithin{equation}{section}
\newtheorem{theorem}{Theorem}
\newtheorem{lemme}[equation]{Lemma}
\newtheorem{proposition}[equation]{Proposition}
\newtheorem{definition}[equation]{Definition}

\newtheorem*{corollary*}{Corollary}

\newcommand\build[3]{\mathrel{\mathop{\kern
0pt#1}\limits_{\textstyle #2}^{\textstyle #3}}}

\newcommand\PP{{\mathbbm P}}

\newcommand\Gal{{\rm Gal}}

\newcommand\hookr\hookrightarrow
\newcommand\hookl\hookleftarrow

\newcommand\isom{\,\smash{\mathop{\hbox to 5mm{\rightarrowfill}}
    \limits^\sim}\,}

\theoremstyle{definition}

\newtheorem{remark}[equation]{Remark}

\newcommand{\mat}[4]{
 \left(  \begin{matrix} #1 & #2 \\ #3 & #4 \end{matrix} \right)}

\def\bP{{\PP}}

\def\P{\textsf{P}}

\def\P1min2{\bP^1_{\backslash\{0,\infty\}}}

\begin{document}
    \title{Local Indecomposability  of Hilbert Modular Galois Representations}
    \author{Bin Zhao\footnote{The author is partially supported by Hida's NSF grant DMS 0753991 and DMS 0854949 through UCLA graduate division.}}

    \date{\today}

\maketitle
\begin{abstract}
We prove the indecomposability of Galois representation restricted
to the $p$-decomposition group attached to a non CM nearly
$p$-ordinary weight two Hilbert modular form under mild conditions.
\end{abstract}

The main purpose of this paper is to decide the indecomposability of
a Hilbert modular ordinary $p$-adic Galois representation restricted
to the decomposition group at $p$, under the assumption that the
representation is not of CM type. This question was originally posed
by R.Greenberg. In the elliptic modular case, it was studied by
Ghate and Vatsal and they gave an affirmative answer in \cite{GV}
under some conditions. In a recent preprint \cite{BGV}, joint with
Balasubramanyam, they generalized their result
 to the Hilbert modular case under some restrictive conditions. Their method is to study the
 specialization of $\Lambda$-adic forms corresponding to weight one
 classical forms, and they use the density
  of such specializations to
 conclude for higher weight modular forms.

In contrast to \cite{BGV}, our method is geometric and relies on the
study of Galois representations attached to weight two nearly
$p$-ordinary Hilbert modular eigenforms. More precisely, let $F$ be
a totally real field and
 $f$ be a (parallel)
weight two Hilbert  modular form over $F$. For each integral ideal
$\mathfrak{a}$ of $F$, let $T(\mathfrak{a})$ be the Hecke operator
for $\mathfrak{a}$, and assume that for all $\mathfrak{a}$, there
exist $c(\mathfrak{a},f)\in \bar{\mathbb{Q}}$, such that
$T(\mathfrak{a})f=c(\mathfrak{a},f)f$. We denote by $K_f$ the field
generated by all $c(\mathfrak{a},f)$'s over $\mathbb{Q}$, which is
known to be a number field.

For any prime $\lambda$ of $K_f$ over a rational prime $p$, let
$K_{f,\lambda}$ be the completion of $K_f$ at $\lambda$. It is well
known that  there is a Galois represention $\rho_f:
\Gal(\bar{\mathbb{Q}}/F)\rightarrow GL_2(K_{f,\lambda})$ attached to
$f$. Moreover if the eigenform $f$ is nearly $p$-ordinary in the
sense that $c(\mathfrak{p},f )$ is a unit mod $\lambda$ for all
primes $\mathfrak{p}$ of $F$ over $p$ (\cite{Wi88} or \cite{H06}
Chapter $3$), then up to equivalence the restriction of $\rho_f$ to
the decomposition group $D_\mathfrak{p}$ of
$\Gal(\bar{\mathbb{Q}}/F)$ at $\mathfrak{p}$ is of the shape (see
\cite{Wi88} Theorem $2$ for the ordinary case and \cite{H89}
Proposition $2.3$ for the nearly ordinary case):
$$\rho_f|_{D_\mathfrak{p}}\sim \mat{\epsilon_1}{\ast}{0}{\epsilon_2}.$$
In this paper we need to put the following technical condition on
$f$ when the degree of $F$ over $\mathbb{Q}$ is even: there exists a
finite place $v$ of $F$ such that $\pi_v$ is square integrable (i.e.
special or supercuspidal) where $\pi_f=\otimes_v \pi_v$ is the
automorphic representation of $GL_2(F_{\mathbb{A}_f})$ associated to
$f$. Then the main result of this paper is:
\begin{theorem}\label{thm}
If $f$ does not have complex multiplication, then
$\rho_f|_{D_\mathfrak{p}}$ is indecomposable.
\end{theorem}
We will state this theorem in a little more general way as Theorem
\ref{Th4} in Section $5$ and give a proof there. Under our
assumption on $f$, there exists a quaternion algebra $B$ over $F$
and a new form $g$ on $B$ such that if $\pi_g=\otimes_v
\tilde{\pi}_v$ is the automorphic representation of
$B_\mathbb{A}^\times=(B\otimes_\mathbb{Q}\mathbb{A}_f)^\times$
associated to $g$, then $\pi_v\cong \tilde{\pi}_v$ for all finite
places $v$ of $F$ which splits in $B$. Then from \cite{H81} Theorem
$4.4$, we can find an abelian variety $A_{f/F}$ and a homomorphism
$K_f'\rightarrow \mathrm{End}^0(A_{f/F})$ attached to $g$, where
$K_f'$ is a finite extension of the Hecke field $K_f$ whose degree
equals to the dimension of $A_f$, such that the Galois
representation $\rho_f$ comes from the $p$-adic  Tate  module of
$A_f$. Hence the problem is purely geometric. The key ingredient in
the proof is a recent result by Hida on local indecomposability of
Tate modules on certain abelian varieties (\cite{H12}). We need more
notations before we state his result. Let $A_{/k}$ be an abelian
variety over a number field $k$ and
$\iota:\mathcal{O}_L\rightarrow\mathrm{End}(A_{/k})$ be a
homomorphism where $\mathcal{O}_L$ is the integer ring of a totally
real field $L$ whose degree over $\mathbb{Q}$ equals to the
dimension of $A_{/k}$. Assume further that there exists a prime
$\mathfrak{P}$ of $k$ over a rational prime $p$  unramified in $k$
and $L$ such that $A_{/k}$ has good reduction at $\mathfrak{P}$, and
there exists a prime $\mathfrak{p}$ of $L$ over $p$ such that the
reduction $A_\mathfrak{P}$ of $A$ at $\mathfrak{P}$ has nontrivial
$\mathfrak{p}$-torsion points. Then Hida's result can be stated as
follows: if the abelian variety $A_{/k}$ does not have complex
multiplication, then the Tate module $T_\mathfrak{p}(A)$ is
indecomposable as an $I_\mathfrak{P}$-module, here $I_\mathfrak{P}$
is the inertia group of $\Gal(\bar{\mathbb{Q}}/k)$ at
$\mathfrak{P}$.

Hida's result has intimate relationship with the problem we want to
work on. But we cannot apply his result directly to the study of
Hilbert modular Galois representations mainly due to the following
two reasons:
\begin{enumerate}
\item The Hecke field $K_f$ may be a CM field instead of a
totally real field, and in this case the field $K_f'$ cannot be
totally real. In section $1$ we give a careful study of the
endomorphism algebras of abelian varieties of $GL_2$-type, and prove
that we can always replace $K_f'$ by a totally real field with the
same degree over $\mathbb{Q}$. Later (in Section $5$) we will prove
that this change of fields does not affect our study of local
indecomposability.
\item In Hida's result, it is assumed that the rational prime $p$ is
unramified in the fields $k$ and $L$. But in the Hilbert modular
case, we have no control on the ramification of $p$ in the Hecke
field $K_f$ and the rational prime $p$ could ramify in the totally
real field $F$. The assumption that $p$ is unramified in $L$ is to
guarantee that the abelian variety sits in the Hilbert modular
Shimura variety for $GL_{2/L}$. In particular it is required that
the Lie algebra $\mathrm{Lie}(A_{/\mathcal{O}_{(\mathfrak{P})}})$ is
free of rank one over
$\mathcal{O}_L\otimes_\mathbb{Z}\mathcal{O}_{(\mathfrak{P})}$
($\mathcal{O}_{(\mathfrak{P})}$ is the localization of
$\mathcal{O}_k$ at the prime $\mathfrak{P}$). When $p$ is ramified
in $L$, this condition may fail on the special fiber of
$\mathcal{O}_{(\mathfrak{P})}$. Thus we change the definition of
$p$-integral models of Hilbert modular Shimura varieties by the one
studied by Deligne and Pappas in \cite{DP}. We will recall the
notion of abelian variety with real multiplication we want to work
with in section $2$, and give the precise definition of the Hilbert
modular Shimura variety in section $3$. We also prove that under
these definitions, the abelian variety $A_{f/F}$ gives rise to a
point in the Shimura variety. The assumption that $p$ is unramified
in the field $k$ is only used in Hida's construction of
eigen-coordinates (see \cite{H12} section $4$). We prove in section
$4$ that his construction works well even if $p$ is ramified in the
field $k$.
\end{enumerate}
We prove the main result in Section $5$, and we give an
$\Lambda$-adic version of our result by applying an argument in
\cite{GV},which generalizes the result of \cite{BGV} unconditionally
when the degree $[F:\mathbb{Q}]$ is odd and when the degree
$[F:\mathbb{Q}]$ is even, we need to assume that at some finite
place $v$ of $F$, the representation $\pi_v$ is square integrable
(see the explanation above Theorem \ref{thm}). At the end we explain
how our result can be applied to study a problem of Coleman.
\\

$\mathbf{Acknowledgement}$: I would like to express my sincere gratitude to my advisor Professor Hida Haruzo
for his patient guidance as well as constant support. Without his careful explanation of his work \cite{H12}
, I cannot get the approach to the work in this paper. I would also like to  thank Ashay Burungale for many
helpful discussions on this topic.
\\

$\mathbf{Notations}$: Throughout this paper, we use
 $F$ to denote a totally real field with degree $d$ over $\mathbb{Q}$
and use $\mathcal{O}_F$ to denote its integer ring. Let  $\mathcal
{D}=\mathcal {D}_{F/\mathbb{Q}}$ be the different of $F/\mathbb{Q}$
and $d_F=Norm_{F/\mathbb{Q}}(\mathcal {D})$ be its discriminant. For
any prime $\mathfrak{p}$ of $\mathcal{O}_F$, let
$\mathcal{O}_\mathfrak{p}$ (resp. $F_\mathfrak{p}$) be the
completion of $\mathcal{O}_F$ (resp. $F$) with respect to
$\mathfrak{p}$.

Fix an algebraic closure $\bar{\mathbb{Q}}$ of $\mathbb{Q}$. For
each rational prime $p$, we fix an algebraic closure
$\bar{\mathbb{Q}}_p$ of $\mathbb{Q}_p$ and an embedding
$i_p:\bar{\mathbb{Q}}\rightarrow\bar{\mathbb{Q}}_p$. We also fix
  an algebraic closure $\bar{\mathbb{F}}_p$ of the prime field
$\mathbb{F}_p$. Let $W_p=W(\bar{\mathbb{F}}_p)$ be the ring of Witt
vectors of $\bar{\mathbb{F}}_p$, and $K_p$ be its quotient field.
Then $K_p$ can be identified with the maximal unramified extension
in $\bar{\mathbb{Q}}_p/\mathbb{Q}_p$.

We will recall the other necessary notations at the beginning of each section.

\tableofcontents

\section{Abelian Varieties of GL(2)-type}
Let $E$ be a number field with degree $d$ over $\mathbb{Q}$, and
$A_{/\overline{\mathbb{Q}}}$ be an abelian variety of dimension $d$.
Set
$\mathrm{End}^0(A_{/\overline{\mathbb{Q}}})=\mathrm{End}(A_{/\overline{\mathbb{Q}}})\otimes_{\mathbb{Z}}\mathbb{Q}$, which
is a finite dimensional semisimple algebra over $\mathbb{Q}$.
Suppose that we have an algebra homomorphism $E\rightarrow
\mathrm{End}^0(A_{/\overline{\mathbb{Q}}})$, which identifies $E$
with a subfield of $\mathrm{End}^0(A_{/\overline{\mathbb{Q}}})$.
Recall that the abelian variety $A_{/\overline{\mathbb{Q}}}$ has complex multiplication
if $\mathrm{End}^0(A_{/\overline{\mathbb{Q}}})$
contains a commutative semisimple subalgebra of dimension $2d$ over $\mathbb{Q}$.
Then from \cite{GME} Section $5.3.1$, we have the following two
results:

\begin{proposition}\label{1.1}
If $A_{/\overline{\mathbb{Q}}}$ does not have complex multiplication,
 then $A_{/\overline{\mathbb{Q}}}$ is isotypic (i.e.
there exists a simple abelian variety $B_{/\overline{\mathbb{Q}}}$
such that $A_{/\overline{\mathbb{Q}}}$ is isogeneous to
$(B_{/\overline{\mathbb{Q}}})^e$ for some $e\geq1$), and
$\mathrm{End}^0_E(A_{/\overline{\mathbb{Q}}})=E$.
\end{proposition}

\begin{proposition}\label{1.2}
Under the conditions of Proposition \ref{1.1}, if we assume further
that $A_{/\overline{\mathbb{Q}}}$ is simple,  then one of the
following four possibilities holds for $D=
\mathrm{End}^0(A_{/\overline{\mathbb{Q}}})$:

\begin{enumerate}
\item $E$ is a quadratic extension of a totally real field $Z$ and
$D$ is a totally indefinite division quaternion algebra over $Z$;
\item $E$ is a quadratic extension of a totally real field $Z$ and
$D$ is a totally definite division quaternion algebra over $Z$;
\item $E$ is a quadratic extension of a CM field $Z$ and
$D$ is a division quaternion algebra over $Z$;
\item $E=D$ and $E$ is totally real.
\end{enumerate}
\end{proposition}

\begin{remark}\label{1.3}
\begin{enumerate}
\item A quaternion algebra $D$ over a totally real field $Z$ is
called totally indefinite if for any real embedding
$\tau:Z\rightarrow\mathbb{R}$, the $\mathbb{R}$-algebra
$D\otimes_{Z,\tau}\mathbb{R}$ is isomorphic to the matrix algebra
$M_2(\mathbb{R})$; the quaternion algebra $D_{/Z}$ is called totally
definite if for any real embedding $\tau:Z\rightarrow\mathbb{R}$,
the $\mathbb{R}$-algebra $D\otimes_{Z,\tau}\mathbb{R}$ is isomorphic
to the Hamilton quaternion algebra $\mathbb{H}$.
\item From Proposition \ref{1.1}, we see that
$\mathrm{End}^0(A_{/\overline{\mathbb{Q}}})$ is always a central
simple algebra and $E$ is a maximal commutative subfield of
$\mathrm{End}^0(A_{/\overline{\mathbb{Q}}})$;
\item As remarked in \cite{GME},  case $2$  in Proposition $1.2$
cannot happen by \cite{Sh1}, Theorem 5(a) and Proposition 15.
\end{enumerate}
\end{remark}

\begin{proposition}\label{1.4}
Under the notations and assumptions in Proposition \ref{1.1}, assume
further that there exists a totally real field $k$ such that the
abelian variety $A_{/\overline{\mathbb{Q}}}$ is defined over $k$,
and the homomorphism $E\rightarrow
\mathrm{End}^0(A_{/\overline{\mathbb{Q}}})$ factors through
$\mathrm{End}^0(A_{/k})$. Then we can find a totally real field $F$
with degree $d$ over $\mathbb{Q}$, which can be embedded into
$D=\mathrm{End}^0(A_{/\overline{\mathbb{Q}}})$ as a subalgebra
sharing the identity with $D$.
\end{proposition}

\begin{proof}
By Propositon \ref{1.1}, we can find a simple abelian variety
$B_{/\overline{\mathbb{Q}}}$ and an integer $e$ such that
$A_{/\overline{\mathbb{Q}}}$ is isogeneous to
$(B_{/\overline{\mathbb{Q}}})^e$. Hence we have an isomorphism of
simple algebras $\mathrm{End}^0(A_{/\overline{\mathbb{Q}}})\cong
M_e(\mathrm{End}^0(B_{/\overline{\mathbb{Q}}}))$,and $d=e\cdot d_1$,
where $d_1$ is the dimension of $B_{/\overline{\mathbb{Q}}}$. Since
any maximal commutative subfield of
$\mathrm{End}^0(A_{/\overline{\mathbb{Q}}})$ has degree $d$ over
$\mathbb{Q}$, any maximal commutative subfield of
$D_1=\mathrm{End}^0(B_{/\overline{\mathbb{Q}}})$ should have
dimension $d/e=d_1$. In other words, we can find number field $E_1$
of degree $d_1$ over $\mathbb{Q}$, which can be embedded into
$\mathrm{End}^0(B_{/\overline{\mathbb{Q}}})$ as a subalgebra. Since
$A_{/\overline{\mathbb{Q}}}$ does not have complex multiplication,
neither does $B_{/\overline{\mathbb{Q}}}$. In summary,
$B_{/\overline{\mathbb{Q}}}$ satisfies all the assumptions in
Proposition \ref{1.2}. Assume that
$\mathrm{End}^0(B_{/\overline{\mathbb{Q}}})$ is of type $3$ as in
Proposition \ref{1.2}, i.e.
$\mathrm{End}^0(B_{/\overline{\mathbb{Q}}})$ is a division
quaternion algebra over a CM field $Z$ and $[E_1:Z]=2$. Since
$d_1=[E_1:\mathbb{Q}]=2[Z:\mathbb{Q}]$, the degree $s$ of $Z$ over
$\mathbb{Q}$ equals to $\frac{d_1}{2}$. Since $Z$ is a CM field, we
can find $s'=\frac{s}{2}$ different embeddings
$\tau_i:Z\rightarrow\overline{\mathbb{Q}} ,i=1,...,s'$, such that
$\mathrm{Hom}_{\mathbb{Q}}(Z,\overline{\mathbb{Q}})=\{\tau_1,...,\tau_{s'},\bar{\tau}_1,...,\bar{\tau}_{s'}\}$,
where $\bar{\tau}_i$ is the complex conjugation of $\tau_i$ for
$i=1,...,s'$. Then we have an isomorphism
$$
\theta:D_1\otimes_\mathbb{Q}\mathbb{R}\cong
\prod_{\tau_i,i=1,...,s'} M_2(\mathbb{C}).
$$

Let $\pi_i$ be the composition
$$D_1\hookrightarrow
D_1\otimes_{\mathbb{Q}}\mathbb{R}\xrightarrow[]{\theta}\prod_{\tau_i,i=1,...,s'}
M_2(\mathbb{C})\xrightarrow[]{\pi_i} M_2(\mathbb{C}),$$
 where the  map $\pi_i$ is
the $i$-th projection,  for  $i=1,...,s'$. Let $\bar{\pi}_i$ be the
complex conjugation of $\pi_i$. Then
$\{\pi_1,...,\pi_{s'},\bar{\pi}_1,...,\bar{\pi}_{s'}\}$ are all the
absolutely irreducible (complex) representations of $D_1$ (up to
isomorphism).

On the other hand, we have a representation of $D_1$ by
$\rho_1:D_1\rightarrow
\mathrm{End}_{\mathbb{C}}(\mathrm{Lie}(B)\otimes_{\bar{\mathbb{Q}}}\mathbb{C})$.
Let $r_i$ (resp. $s_i$) be the multiplicity of $\pi_i$ (resp
$\bar{\pi}_i$) in $\rho_1$. Then for any $z\in Z$, the trace of
$\rho_1(z)$ is given by the formula:
$$
\mathrm{Tr}(\rho_1(z))=2\sum_{i=1}^{s'}(r_i\tau_i(z)+s_i\bar{\tau}_i(z)).
$$
Since
$\mathrm{Lie}(A_{/\bar{\mathbb{Q}}})\cong(\mathrm{Lie}(B_{/\bar{\mathbb{Q}}}))^e$,
we have the representation
$\rho:D\rightarrow\mathrm{End}_{\mathbb{C}}(\mathrm{Lie}(A)\otimes_{\bar{\mathbb{Q}}}\mathbb{C})$,
such that for any $z\in Z$,
$$
\mathrm{Tr}(\rho(z))=e\mathrm{Tr}(\rho_1(z))=2e\sum_{i=1}^{s'}(r_i\tau_i(z)+s_i\bar{\tau}_i(z)).
$$
Since $Z\subseteq E$ and the homomorphism
$E\rightarrow\mathrm{End}^0(A_{/\bar{\mathbb{Q}}})$ factors through
$\mathrm{End}^0(A_{/k})$, we have $\mathrm{Tr}(\rho(z))\in k$, for
any $z\in Z$. From \cite{Sh1} Section 4, we have $r_i+s_i=2$, for
all $i=1,...,s'$. Thus for each $i$, either $r_i=s_i=1$ or $r_i\cdot
s_i=0$. If $r_i\cdot s_i=0$ for at least one $i$, then
$\mathrm{Tr}(\rho(z))$ cannot lie in the totally real field $k$ for
all $z\in Z$ as $Z$ is assumed to be a CM field. Hence $r_i=s_i=1$
for all $i$. Then by \cite{Sh1} Theorem $5(e)$ and Proposition $19$,
this case cannot happen.

Combined with Remark \ref{1.3}(3),  we see that
$\mathrm{End}^0(B_{/\overline{\mathbb{Q}}})$ is either a totally
real field or a totally indefinite division algebra over a totally
real field. Then the existence of $F$ results from:

\begin{lemme}
Let $D$ be a central simple algebra over a totally real field $Z$
with $[D:Z]=d^2$. If for all real embeddings
$\tau:Z\rightarrow\mathbb{R}$, the $\mathbb{R}$-algebra
$D\otimes_{Z,\tau}\mathbb{R}$ is isomorphic to the matrix algebra
$M_d(\mathbb{R})$, then we can find a field extension $F/Z$ with
degree $d$ such that $F$ is totally real and can be embedded into
$D$ as an $Z$-subalgebra.
\end{lemme}
Proof of the lemma: We use an argument similar with the proof of
Lemma $1.3.8$ in \cite{CM}. It is enough to find a field extension
$F/Z$ with degree $d$ such that $F$ is totally real and splits
$D$(i.e $D\otimes_Z F\cong M_d(F)$).

Let $\Sigma$ be a non empty set of non-archimedean places of $Z$
containing all the finite places where $D$ does not split, and
$\Sigma_\infty$ be the set of archimedean places of $Z$. By the weak
approximation theorem, the natural map:
$$
Z\rightarrow\prod_{v\in \Sigma}Z_v\times
\prod_{v\in\Sigma_\infty}Z_v
$$
has dense image. Hence we can find a monic polynomial $f(X)\in Z[X]$
of degree $d$, such that it is sufficiently close to a monic
irreducible polynomial of degree $d$ over $Z_v$ for all $v\in
\Sigma$, and it is sufficiently close to a totally split polynomial
of degree $d$ over $\mathbb{R}$ for all $v\in \Sigma_\infty$. Set
$F=Z[X]/(f(X))$. Then $F/Z$ is a degree $d$ field extension such
that $F$ is totally real and for any $v\in \Sigma$, there is exactly
one place $w$ of $F$ lying over $v$ and hence $F_w/Z_v$ is a degree
$d$ extension of local fields.

We still need to check that $F$ splits $D$. Since $D\otimes_Z F$ is
a central simple algebra over $F$ and $F$ is a global field, it is
enough to prove that for any place $w$ of $F$ (archimedean and
non-archimedean), we have an isomorphism $D\otimes_Z F_w\cong
M_d(F_w)$. Let $v$ be the place of $Z$ over which $w$ lies.

If $w$ is archimedean, then $Z_v\cong F_w\cong\mathbb{R}$, and hence
\begin{equation}
D\otimes_Z F_w\cong (D\otimes_Z Z_v)\otimes_{Z_v}F_w\cong
M_d(\mathbb{R}),
\end{equation}
by our assumption on $D$.

If $w$ is non-archimedean and $v$ is not in  $\Sigma$,then
$D\otimes_Z Z_v$ is already isomorphic to the matrix algebra over
$Z_v$, so we are safe in this case.

Finally, assume that $w$ is non-archimedean and $v\in \Sigma$. As
$F_w/Z_v$ is a degree $d$ extension of local field, the base change
from $Z_v$ to $F_w$ induces a homomorphism of Brauer groups
$\mathrm{Br}(Z_v)\rightarrow \mathrm{Br}(F_w)$, which under the
isomorphism $\mathrm{Br}(Z_v)\cong \mathrm{Br}(F_w)\cong
\mathbb{Q}/\mathbb{Z}$ by local class field theory, is nothing but
multiplication by $d$. As $[D:Z]=d^2$, the order of $D\otimes_Z Z_v$
in $\mathrm{Br}(Z_v)$ is divisible by $d$. This implies that
$D\otimes_Z F_w$ represents the identity element in
$\mathrm{Br}(F_w)$; i.e.$D\otimes_Z F_w\cong M_d(F_w)$. Hence $F/Z$
is the desired extension.

\end{proof}
Hereafter we always work with the pair
$(A_{/\overline{\mathbb{Q}}},\iota:F\rightarrow
\mathrm{End}^0(A_{/\overline{\mathbb{Q}}}))$, where $F$ is a totally
real field with degree $d$ over $\mathbb{Q}$. Since the abelian
variety $A_{/\overline{\mathbb{Q}}}$ is projective, we can find a
number field $k$ such that $A$ is defined over $k$, and
$\mathrm{End}(A_{/\overline{\mathbb{Q}}}))=\mathrm{End}(A_{/k})$.
Let $\mathcal{O}_k$ be the integer ring of $k$, and for all prime
ideals $\mathfrak{P}$
 of $\mathcal{O}_k$ over some rational prime $p$, let
$\mathcal{O}_{(\mathfrak{P})}$ be localization of $\mathcal{O}_k$ at
the prime $\mathfrak{P}$ and
$\mathbb{F}_{\mathfrak{P}}=\mathcal{O}_k/\mathfrak{P}$ be its
residue field. As in \cite{H12}, we make the following assumption:

(NLL) the abelian variety $A_{/k}$ has good reduction at
$\mathfrak{P}$ and the reduction
$A_0=A\otimes_{\mathcal{O}_{\mathfrak{P}}}\mathbb{F}_{\mathfrak{P}}$
has nontrivial $\textit{p}$-torsion $\bar{\mathbb{F}}_p$-points.

Change the abelian variety $A_{/k}$ if necessary, we can assume that
$\iota$ gives a homomorphism
$\iota:\mathcal{O}_F\rightarrow\mathrm{End}(A_{/k})$. Let
$\bar{\mathbb{F}}_p$ be an algebraic closure of
$\mathbb{F}_{\mathfrak{P}}$. $W_p=W(\bar{\mathbb{F}}_p)$ is the ring
of Witt vectors of $\bar{\mathbb{F}}_p$. We have the decomposition
of Barsotti-Tate $\mathcal{O}_F$-modules:

$$A_0[p^\infty]=\bigoplus_{\mathfrak{p}|p}A_0[\mathfrak{p}^\infty].$$
Here $\mathfrak{p}$ ranges over the primes ideals of $\mathcal{O}_F$
over $p$ and for each $\mathfrak{p}$, let
$$A_0[\mathfrak{p}^\infty]=\lim_{\rightarrow}A_0[\mathfrak{p}^n]$$
be  the $\mathfrak{p}$-divisible Barsotti-Tate group of $A_0$. We
also define
$$T_\mathfrak{p}(A_0)=\lim_{\leftarrow}A_0[\mathfrak{p}^n](\bar{\mathbb{F}}_p)$$
as the $\mathfrak{p}$-divisible Tate module of $A_0$.

We say that a prime $\mathfrak{p}$ of $\mathcal{O}_F$ over $p$ is
ordinary if $A_0[\mathfrak{p}]$ has nontrivial
$\bar{\mathbb{F}}_p$-points, otherwise we say that $\mathfrak{p}$ is
local-local. When $\mathfrak{p}$ is ordinary and $p$ is unramified
in $k$, we have an exact sequence of Barsotti-Tate
$\mathcal{O}_{\mathfrak{p}}$-modules over $W_p$:

$$0\rightarrow\mu_{p^\infty}\otimes_{\mathbb{Z}_p}\mathcal{O}_{\mathfrak{p}}^\ast\rightarrow A[\mathfrak{p}^\infty]_{/W_p}\rightarrow F_\mathfrak{p}/\mathcal{O}_{\mathfrak{p}}\rightarrow 0.$$
Here $\mathcal{O}_{\mathfrak{p}}^\ast=\mathrm{Hom}_{\mathbb{Z}_p}(\mathcal{O}_{\mathfrak{p}},\mathbb{Z}_p)$ is the $\mathbb{Z}_p$-dual of $\mathcal{O}_{\mathfrak{p}}$.

Let $\Sigma^{ord}_p$ be the set of all ordinary primes of
$\mathcal{O}_F$ over $p$, and $\Sigma^{ll}_p$ be the set of all
local-local primes. Then the condition (NLL) is equivalent to the
fact that  $\Sigma^{ord}_p$ is not empty. Also we define:

$$A_0[p^\infty]^{ord}=\bigoplus_{\mathfrak{p}\in \Sigma_p^{ord}}A_0[\mathfrak{p}^\infty],A_0[p^\infty]^{ll}=\bigoplus_{\mathfrak{p}\in \Sigma_p^{ll}}A_0[\mathfrak{p}^\infty].$$

\section{Abelian varieties with real multiplication}

In this section we introduce the notion of abelian varieties with
real multiplication (AVRM for short). Then we prove that by a change
via an isogeny, we can make the abelian variety $A_{/k}$ considered
in the previous section into an abelian variety of this type.

Fix an invertible $\mathcal{O}_F$-module $L$, with a notion of
positivity $L_+$ on it: for each real embedding $\tau: F\rightarrow
\mathbb{R}$, we give an orientation on the line
$L\otimes_{\mathcal{O}_F,\tau}\mathbb{R}$. First we recall the
following definition in \cite{DP}:
\begin{definition}\label{2.1}
An $L$-polarized abelian scheme with real multiplication by
$\mathcal{O}_F$ is the triple $(A_{/S},\iota,\varphi)$ consisting of

\begin{enumerate}
\item $A_{/S}$ is an abelian scheme of relative dimension $d$;
\item $\iota:\mathcal{O}_F\rightarrow \mathrm{End}(A_{/S})$ is an algebra
homomorphism which gives $A_{/S}$ an $\mathcal{O}_F$-module
structure;
\item $\varphi:\underline{L}\rightarrow
\mathrm{Hom}^{Sym}_{\mathcal{O}_F}(A_{/S},A^t_{/S})$ is an
$\mathcal{O}_F$-linear morphism of sheaves of
$\mathcal{O}_F$-modules on the \'{e}tale site
$(Sch_{/S})_{\acute{e}t}$ of the category of $S$-schemes, such that
$\varphi$ sends totally positive elements of $L$ into polarizations
of $A_{/S}$, and the natural morphism
$\alpha:A\otimes_{\mathcal{O}_F}\underline{L}\rightarrow A^t$ is an
isomorphism. Here $A^t$ is the dual abelian scheme of $A$, and
$\underline{L}$ is the constant sheaf valued in $L$, and the sheaf
$\mathrm{Hom}^{Sym}_{\mathcal{O}_F}(A_{/S},A^t_{/S})$ is defined by
:

$(Sch_{/S})_{\acute{e}t} \ni T \mapsto
\mathrm{Hom}^{Sym}_{\mathcal{O}_F,T}(A_{T_{/T}},A^t_{T_{/T}})=\{\lambda:A_{T_{/T}}\rightarrow
A^t_{T_{/T}}|\lambda  $ is $  \mathcal{O}_F$-linear and
symmetric$\}$
\end{enumerate}
When $L=\mathfrak{c}$ is a fractional ideal of $\mathcal{O}_F$ with
the natural notion of positivity, we call the isomorphism $\alpha:
A\otimes_{\mathcal{O}_F}\mathfrak{c}\rightarrow A^t$ a
$\mathfrak{c}$-polarization of $A$ (see \cite{K1}1.0 for more
discussion). We also make the convention that for $c\in
\mathfrak{c}$, the morphism $\lambda(c):A\rightarrow A^t$ is the
corresponding symmetric $\mathcal{O}_F$-linear homomorphism.
\end{definition}

\begin{remark}\label{2.2}
The fppf abelian sheaf $A\otimes_{\mathcal{O}_F}\underline{L}$ is
the sheafication of the functor $$(Sch_{/S})_{\textrm{fppf}}\ni T
\mapsto A(T)\otimes_{\mathcal{O}_F}L.$$

This sheaf is represented by an abelian scheme over $S$, which is
denoted by $A\otimes_{\mathcal{O}_F}L$. Hence the isomorphism
$\alpha$ in (3) can be regarded as an isomorphism of abelian schemes
over $S$.
\end{remark}

\begin{definition}\label{2.3}
Let $A_{/S}$ be an abelian scheme over a scheme $S$ of relative
dimension $d$, and $\iota:\mathcal{O}_F\rightarrow
\mathrm{End}(A_{/S})$ be an algebra homomorphism. We say that the
pair $(A_{/S},\iota)$ satisfies the condition $(\mathrm{DP})$ if the
natural morphism
$\alpha:A\otimes_{\mathcal{O}_F}\mathrm{Hom}^{Sym}_{\mathcal{O}_F}(A_{/S},A^t_{/S})\rightarrow
A^t$ is an isomorphism. We say that the pair $(A_{/S},\iota)$
satisfies the condition $(\mathrm{RA})$ if Zariski locally on $S$,
$\mathrm{Lie}(A_{/S})$ is a free
$\mathcal{O}_S\otimes_\mathbb{Z}\mathcal{O}_F$-module of rank $1$.
\end{definition}

We remark here that the two conditions $(\mathrm{DP})$ and
$(\mathrm{RA})$ in Definition \ref{2.3} can be checked at each
geometric point of the base scheme $S$. When the pair
$(A_{/S},\iota)$ satisfies the condition $(\mathrm{RA})$, we come to
the notion of abelian schemes with real multiplication (by
$\mathcal{O}_F$) defined in \cite{Ra}. As explained in \cite{DP}2.9,
when $d_F$ is invertible on $S$, condition (DP) in Definition
\ref{2.3} implies (RA). For later use, we explain that condition
(RA) implies (DP) under some assumption on $S$ and by a suitable
choice of the pair $(L,L_+)$, we can make $A_{/S}$ be an
$L$-polarized abelian scheme with real multiplication by
$\mathcal{O}_F$ . First we need the following:

\begin{proposition}\label{2.4}
(\cite{Ra}$\textrm{1.17,1.18}$) Let $A_{/S}$ be an abelian scheme of
relative dimension $d$,and $\iota:\mathcal{O}_F\rightarrow
\mathrm{End}(A_{/S})$ be an algebra homomorphism. Then the \'{e}tale
sheaf $\mathrm{Hom}^{Sym}_{\mathcal{O}_F}(A_{/S},A^t_{/S})$ defined
above is locally constant with values in a projective
$\mathcal{O}_F$-module of rank $1$, endowed with a notion of
positivity corresponding to polarizations of $A_{/S}$. In
particular, when $S$ is normal and connected, this sheaf is
constant.
\end{proposition}

Here we remark that in \cite{Ra}, the abelian scheme $A_{/S}$ is
assumed to satisfy condition $(\mathrm{RA})$. But this condition is
not necessary in the proof of the above proposition.

Now assume that $S$ is normal and connected (e.g. S is the spectrum
of the integer ring of a number field). Then from Proposition
\ref{2.4} we can find a projective $\mathcal{O}_F$-module $\mathcal
{M}$ of rank $1$  with a notion of positivity $\mathcal {M}_+$ and a
morphism $\varphi:\underline{\mathcal {M}}\rightarrow
\mathrm{Hom}^{Sym}_{\mathcal{O}_F}(A_{/S},A^t_{/S})$. To check this
$\varphi$ satisfies condition $(3)$ in Definition \ref{2.1}, we
still need to check that the morphism
$\alpha:A\otimes_{\mathcal{O}_F}\underline{\mathcal {M}}\rightarrow
A^t$ is an isomorphism.

We can assume that $S=\mathrm{Spec}(k)$, where $k$ is an separably
closed field and we want to prove that $\alpha$ is an isomorphism of
abelian varieties over $k$. Then it suffices to show that for any
rational prime $l$, there exists $0\neq \lambda\in \mathcal {M}$,
such that $\deg(\varphi(\lambda))$ is prime to $l$. In fact, for any
$\alpha\in \mathcal {M}$, we have a natural morphism $A\rightarrow
A\otimes_{\mathcal{O}_F}\mathcal {M}$ whose effect on $R$-valued
points is given by the formula ($R$ is an $k$-algebra):
$$A(R)\ni a \mapsto
a\otimes_{\mathcal{O}_F}\lambda\in
A(R)\otimes_{\mathcal{O}_F}\mathcal {M}.$$
 The composition of this
morphism with $\alpha$ is $\varphi(\lambda)$. Hence
$\deg(\alpha)|\deg(\varphi(\lambda))$. In particular, $\deg(\alpha)$
is prime to $l$. As $l$ is arbitrary, $\deg(\alpha)=1$ and hence
$\alpha$ is an isomorphism.

To prove the existence of $\lambda$, we apply an argument in
\cite{Goren} Chapter $3$ Section $5$: when $char(k)>0$, by
\cite{Ra}1.13, we can always lift the pair $(A_{/k},\iota)$ to an
abelian scheme with real multiplication
$(\tilde{A}_{/W(k)},\tilde{\iota})$ satisfying $(\mathrm{RA})$. Here
$W(k)$ is the ring of Witt vectors of $k$. Hence we can assume that
$char(k)=0$. By Lefschetz principle, we can assume that $k$ is the
complex filed. Then the existence of $\lambda$ follows from the
complex uniformization \cite{Goren}Chapter $2$ Section $2.2$.

The following proposition tells us that when $S$ is a scheme of
characteristic $0$, condition $(\mathrm{RA})$ and hence
$(\mathrm{DP})$ is automatically satisfied.

\begin{proposition}\label{2.5}
Let $k$ be a field of characteristic $0$, $A_{/k}$ be an abelian
variety of dimension $d$, and $\iota:\mathcal{O}_F\rightarrow
\mathrm{End}(A_{/k})$ be an algebra homomorphism. Then
$\mathrm{Lie}(A_{/k})$ is a free
$\mathcal{O}_F\otimes_{\mathbb{Z}}k$-module of rank 1.
\end{proposition}
\begin{proof}
By Lefschetz principle we can again work over the complex field.
Then the result follows from \cite{Goren} Chapter 2, Corollary 2.6.
\end{proof}

Now we consider the object considered in Section $1$. Let $k$ be a
number field, $A_{/k}$ be an abelian variety of dimension $d$
satisfying the condition $(\mathrm{NLL})$, and $\iota: F\rightarrow
\mathrm{End}^0(A_{/k})$ be an algebra homomorphism. We want to prove
that there is a $\mathfrak{c}$-polarized abelian variety
$A'_{/\mathcal{O}_k}$ with real multiplication by $\mathcal{O}_F$
which is isogenous to $A_{/k}$.

We can find an order $\mathcal{O}$ in $F$ which is mapped into
$\mathrm{End}(A)$ under $\iota$. By Serre's Tensor construction
(\cite{CM}1.7.4.), we can find an isogeny $f: A\rightarrow A'$ over
$k$, and the induced isomorphism $\mathrm{End}^0(A_{/k})\rightarrow
\mathrm{End}^0(A'_{/k})$ carries $\mathcal{O}_F\subseteq
\mathrm{End}^0(A_{/k})$ into $\mathrm{End}(A'_{/k})$. Hence we have
an algebra homomorphism $\iota':\mathcal{O}_F\rightarrow
\mathrm{End}(A'_{/k})$. By our assumption, $A_{/k}$ has good
reduction at the prime $\mathfrak{P}$ of $\mathcal{O}_k$. By the
criterion of N\'{e}ron-Ogg-Shafarevich (\cite{ST} Section 1
Corollary 1), $A'_{/k}$ also has good reduction at $\mathfrak{P}$,
and hence can be extended to an abelian scheme
$A'_{/\mathcal{O}_{(\mathfrak{P})}}$ (recall that
$\mathcal{O}_{(\mathfrak{P})}$ is the localization of
$\mathcal{O}_k$ at the prime $\mathfrak{P}$). Since
$\mathcal{O}_{(\mathfrak{P})}$ is a normal domain, by a lemma of
Faltings (see \cite{Faltings} Lemma $1$),  the restriction to the
generic fiber induces a bijection
$$\mathrm{End}(A'_{/\mathcal{O}_{(\mathfrak{P})}})\rightarrow \mathrm{End}(A'_{/k}).$$
So we have an algebra homomorphism $\mathcal{O}_F\rightarrow
\mathrm{End}(A'_{/\mathcal{O}_{(\mathfrak{P})}})$, which is again
denoted by $\iota'$.

From Proposition \ref{2.4}, the \'{e}tale sheaf
$\mathrm{Hom}^{Sym}_{\mathcal{O}_F}(A'_{/\mathcal{O}_{(\mathfrak{P})}},A'^t_{/\mathcal{O}_{(\mathfrak{P})}})$
is a constant sheaf $\underline{\mathfrak{c}}$ for some fractional
ideal $\mathfrak{c}$, with the natural notion of positivity
$\mathfrak{c}_+$. Thus we have a natural isomorphism
$\varphi:\underline{\mathfrak{c}}\rightarrow
\mathrm{Hom}^{Sym}_{\mathcal{O}_F}(A'_{/\mathcal{O}_{(\mathfrak{P})}},A'^t_{/\mathcal{O}_{(\mathfrak{P})}})$
which sends totally positive elements of $\mathfrak{c}$ to
polarizations of $A'_{/\mathcal{O}_{(\mathfrak{P})}}$. We still need
to check that the natural morphism $\alpha:
A'\otimes_{\mathcal{O}_F}\underline{\mathfrak{c}}\rightarrow A'^t$
is an isomorphism over $\mathcal{O}_{(\mathfrak{P})}$. As
$char(k)=0$, by Proposition \ref{2.5}, $\alpha$ is an isomorphism at
the generic fiber of $\mathcal{O}_{(\mathfrak{P})}$. Hence $\alpha$
is an isomorphism again by Faltings lemma.

In summary, we have:

\begin{proposition}\label{2.6}
Let $A_{/k}$ be an abelian variety of dimension $d$ satisfying the
condition $(\mathrm{NLL})$ in Section $1$, and $\iota: F\rightarrow
\mathrm{End}^0(A_{/k})$ be an algebra homomorphism. Then we can find
a fractional ideal $\mathfrak{c}$ and an $\mathfrak{c}$-polarized
abelian scheme $(A'_{/\mathcal{O}_{(\mathfrak{P})}},\iota',\varphi)$
with real multiplication by $\mathcal{O}_F$ such that $A_{/k}$ is
$k$-isogenous to $A'_{/k}$.
\end{proposition}

\begin{remark}\label{2.9}
Let $A_{/S}$ be an abelian scheme of relative dimension $d$ and
$\iota:\mathcal{O}_F\rightarrow \mathrm{End}(A_{/S})$ be an algebra
homomorphism. By a similar argument as above, we see that if $S$ is
an integral normal scheme and the generic fiber of $S$ is of
characteristic $0$, then the pair $(A_{/S},\iota)$ must satisfy the
condition $(\mathrm{DP})$.
\end{remark}
For later discussion, we need the following:
\begin{lemme}\label{2.10}
Let $A_{/S},A'_{/S}$ be two abelian schemes of relative dimension
$d$, and $\iota:\mathcal{O}_F\rightarrow
\mathrm{End}(A_{/S}),\iota':\mathcal{O}_F\rightarrow
\mathrm{End}(A'_{/S})$ be two algebra homomorphisms. Suppose that
there exists an $\mathcal{O}_F$-linear \'{e}tale homomorphism of
abelian schemes $f:A\rightarrow A'$. If the pair $(A_{/S},\iota)$
satisfies the condition $(\mathrm{DP})$, so does $(A'_{/S},\iota')$.
\end{lemme}

\begin{proof}
Without loss of generality, we can assume that $S=\mathrm{Spec}(k)$
for some separably closed field $k$. If $char(k)=0$, then
$(A'_{/S},\iota')$ satisfies (DP) automatically by Proposition
\ref{2.5}. So we can assume that $char(k)=p>0$. From the discussion
of \cite{Goren} Page $100-101$, the pair $(A_{/k},\iota)$ can be
lifted to characteristic $0$; i.e., there exist:
\begin{enumerate}
\item a normal local domain $W$ with maximal ideal $\mathfrak{m}$ and
residue field $k$ such that the quotient  field of $W$ is of
characteristic $0$;
\item an abelian scheme $\tilde{A}_{/W}$ with an
$\mathcal{O}_F$-action $\tilde{\iota}:\mathcal{O}_F\rightarrow
\mathrm{End}(\tilde{A}_{/W})$ such that $(A_{/k},\iota)$ is
isomorphic the the pull back of $(\tilde{A}_{/W},\tilde{\iota})$
under the natural morphism $\mathrm{Spec}(k)\rightarrow
\mathrm{Spec}(W)$.
\end{enumerate}
Replacing $W$ by its $\mathfrak{m}$-adic completion if necessary, we
can assume that $W$ is complete.

Since $f:A\rightarrow A'$ is \'{e}tale and $\mathcal{O}_F$-linear,
$C=ker(f)$ is a finite \'{e}tale $\mathcal{O}_F$-submodule of
$A_{/k}$. Then we can lift $C$ to an \'{e}tale
$\mathcal{O}_F$-submodule $\tilde{C}_{/W}$ of $\tilde{A}_{/W}$. Let
$\tilde{A}'_{/W}$ be the quotient of $\tilde{A}_{/W}$ by
$\tilde{C}_{/W}$, with the natural homomorphism
$\tilde{\iota}':\mathcal{O}_F\rightarrow
\mathrm{End}(\tilde{A}'_{/W})$ induced from $\tilde{A}_{/W}$. By the
above construction it is easy to see that
$(\tilde{A}'_{/W},\tilde{\iota}')$ lifts $(A'_{/k},\iota')$. Then
from Remark \ref{2.9}, $(A'_{/k},\iota')$ satisfies (DP).
\end{proof}

\section{Hilbert modular Shimura variety}

Fix a finite set of primes $\Xi$. Set
\begin{displaymath}
\mathbb{Z}_{(\Xi)}=\{\frac{m}{n}\in \mathbb{Q}|m,n\in
\mathbb{Z},(n,p)=1, \forall p\in \Xi\}.
\end{displaymath}
Then define
$\mathcal{O}_{(\Xi)}=\mathcal{O}_F\otimes_\mathbb{Z}\mathbb{Z}_{(\Xi)}$,
and $\mathcal{O}_{(\Xi),+}^\times$ as the set of totally positive
units in $\mathcal{O}_{(\Xi)}$. Also we define:

$$\widehat{\mathbb{Z}}=\underleftarrow{\lim} \mathbb{Z}/n\mathbb{Z},\ \widehat{\mathbb{Z}}^{(\Xi)}=\underleftarrow{\lim}  \mathbb{Z}/n\mathbb{Z},
\ \mathbb{Z}_{\Xi}=\prod_{l\in\Xi}\mathbb{Z}_l,$$ where in the first
inverse limit, $n$ ranges over all positive numbers, and in the
second inverse limit, $n$ ranges over all positive integers prime to
$\Xi$. Let $\mathbb{A}$ be the adele ring of $\mathbb{Q}$. Then set
\begin{displaymath}
\mathbb{A}^{(\Xi\infty)}=\{x\in\mathbb{A}|x_l=x_\infty=0, \forall
l\in \Xi\},
\end{displaymath}
and
$F_{\mathbb{A}^{(\Xi\infty)}}=F\otimes_\mathbb{Q}\mathbb{A}^{(\Xi\infty)}$.

Define the algebraic group
$G=\mathrm{Res}_{\mathcal{O}_F/\mathbb{Z}}(GL(2))$ and let $Z$ be
its center. $K$ is an open compact subgroup of
$G(\widehat{\mathbb{Z}})$ which is maximal at $\Xi$, in the sense
that $K=G(\mathbb{Z}_{\Xi})\times K^{(\Xi)}$, where
$$K^{(\Xi)}=\{x\in K|x_p=1 \text{~for all~} p\in \Xi\}.$$

\begin{definition}\label{def3.1}
Define the functor $\mathscr{E'}_K^{(\Xi)}:
Sch_{/\mathbb{Z}_{(\Xi)}}\rightarrow Sets$, such that for each
$\mathbb{Z}_{(\Xi)}$-scheme $S$,
$\mathscr{E'}_K^{(\Xi)}(S)=[(A_{/S},\iota,\bar{\lambda},\bar{\eta}^{(\Xi)})]$.
Here $[(A_{/S},\iota,\bar{\lambda},\bar{\eta}^{(\Xi)})]$ is the set
of isomorphism classes of quadruples
$(A_{/S},\iota,\bar{\lambda},\bar{\eta}^{(\Xi)})$ consisting of:
\begin{enumerate}
\item an abelian scheme $A_{/S}$ of relative dimension $d$;
\item an algebra homomorphism
$\iota:\mathcal{O}_F\rightarrow\mathrm{End}(A_{/S})$ such that the
pair $(A_{/S},\iota)$ satisfies the condition $(\mathrm{DP})$ (see
Definition \ref{2.3});
\item a subset $\{\lambda\circ\iota(b): b\in
\mathcal{O}_{(\Xi),+}^\times\}$ of
$\mathrm{Hom}(A_{/S},A^t_{/S})\otimes_\mathbb{Z}\mathbb{Q}$, where
$\lambda: A_{/S}\rightarrow A^t_{/S}$ is an $\mathcal{O}_F$-linear
polarization of $A$, whose degree is prime to $\Xi$;
\item  $\bar{\eta}^{(\Xi)}$ is a $K$-rational level structure of the abelian scheme $A_{/S}$ (see Remark \ref{3.2.5} below).
\end{enumerate}
An isomorphism from one quadruple
$(A_{/S},\iota,\bar{\lambda},\bar{\eta}^{(\Xi)})$ to another
$(A'_{/S},\iota',\bar{\lambda}',\bar{\eta}'^{(\Xi)})$ is an element
$f\in
\mathrm{Hom}(A_{/S},A'_{/S})\otimes_\mathbb{Z}\mathbb{Z}_{(\Xi)}$
whose degree is prime to $\Xi$ such that:
\begin{enumerate}
\item $f\circ\iota(b)=\iota'(b)\circ f$ for all $b\in
\mathcal{O}_F$;
\item $f^t\circ \bar{\lambda}'\circ f=\bar{\lambda}$ as subsets of
$\mathrm{Hom}(A_{/S},A^t_{/S})\otimes_\mathbb{Z}\mathbb{Q}$;
\item we have the equality of level stuctures: $V^{(\Xi)}(f)(\bar{\eta}^{(\Xi)})=\bar{\eta}'^{(\Xi)}$ .
\end{enumerate}
\end{definition}

Now we choose a representative $I=\{\mathfrak{c}\}$ of fractional
ideals in the finite class group
$Cl(K)=(F_{\mathbb{A}^{(\Xi\infty)}})^\times/\mathcal{O}_{(\Xi),+}^\times
\det(K)$. For each $\mathfrak{c}$, fix an $\mathcal{O}_F$-lattice
$L_{\mathfrak{c}}\subseteq V=F^2$ such that
$\wedge(L_{\mathfrak{c}}\wedge L_{\mathfrak{c}})=\mathfrak{c}^\ast$.
Here $\wedge: V\wedge V\rightarrow F$ is the alternating form given
by $((a_1,a_2),(b_1,b_2))\mapsto a_1 b_2-a_2 b_1$.

\begin{definition}\label{def3.2}
Define the functor
$\mathscr{E}_{K,\mathfrak{c}}^{(\Xi)}:Sch_{/\mathbb{Z}_{(\Xi)}}\rightarrow
Sets$,such that for each $\mathbb{Z}_{(\Xi)}$-scheme $S$,
$\mathscr{E}_{K,\mathfrak{c}}^{(\Xi)}(S)=\{(A_{/S},\iota,\phi,\bar{\alpha}^{(\Xi)})\}_{/\cong}$,
where $\{(A_{/S},\iota,\phi,\bar{\alpha}^{(\Xi)})\}_{/\cong}$ is the
set of isomorphic classes of quadruples
$(A_{/S},\iota,\phi,\bar{\alpha}^{(\Xi)})$ consisting of
\begin{enumerate}
\item an abelian scheme $A_{/S}$ of relative dimension $d$;
\item an algebra homomorphism
$\iota:\mathcal{O}_F\rightarrow\mathrm{End}(A_{/S})$;
\item a $\mathfrak{c}$-polarization $\phi: A\otimes_{\mathcal{O}_F}\mathfrak{c}\rightarrow A^t$ of
$A_{/S}$ (see Definition \ref{2.1});
\item $\bar{\alpha}^{(\Xi)}$ is a $K$-integral level structure of the abelian scheme $A_{/S}$ (see Remark \ref{3.2.5} below).
\end{enumerate}
An isomorphism from one quadruple
$(A_{/S},\iota,\phi,\bar{\alpha}^{(\Xi)})$ to another
$(A'_{/S},\iota',\phi',\bar{\alpha}'^{(\Xi)})$ is an isomorphism
$f:A\rightarrow A'$ of abelian schemes over $S$ such that
\begin{enumerate}
\item $f\circ\iota(b)=\iota'(b)\circ f$ for all $b\in
\mathcal{O}_F$;
\item $f^t\circ \phi'\circ (f\otimes_{\mathcal{O}_F}Id_{\mathfrak{c}})=\phi: A\otimes_{\mathcal{O}_F}\mathfrak{c}\rightarrow A^t$;
\item we have an equality of integral level structures: $T^{(\Xi)}(f)(\bar{\alpha}^{(\Xi)})=\bar{\alpha}'^{(\Xi)}$.
\end{enumerate}
\end{definition}

\begin{remark}\label{3.2.5}
Here we briefly recall the notion of level structures on an abelian
scheme with real multiplication. As in Definition \ref{def3.1}
 and \ref{def3.2}, we fix an abelian scheme $A_{/S}$ and a homomorphism $\iota:\mathcal{O}_F\rightarrow\mathrm{End}(A_{/S})$. Take a point
$s\in S$ and let $\bar{s}:\mathrm{Spec}(k(\bar{s}))\rightarrow S$ be
a geometric point of $S$ over $s$, where $k(\bar{s})$ is a separably
closed field extension of the residue field $k(s)$ of S at the point
$s$. Consider the prime-to-$\Xi$   Tate module
$$T^{\Xi}(A_{\bar{s}})=\lim_{\longleftarrow_N}A[N](k(\bar{s})),$$
where $N$ runs through all positive integers prime to $\Xi$, and set
$V^{\Xi}(A_{\bar{s}})=T^{\Xi}(A_{\bar{s}})\otimes_\mathbb{Z}\mathbb{Z}_{\Xi}$,
which is a free $F_{\mathbb{A}^{(\Xi\infty)}}$-module of rank $2$.
When $N$ is invertible on $S$, the finite scheme $A[N]$ is \'{e}tale
over $S$. The algebraic fundamental group $\pi(S,\bar{s})$ acts on
$A[N](k(\bar{s}))$, and hence on $T^{\Xi}(A_{\bar{s}})$ and
$V^{\Xi}(A_{\bar{s}})$. This action is compatible with the action of
$G(\widehat{\mathbb{Z}}^{(\Xi)})$ (resp.
$G(F_{\mathbb{A}^{(\Xi\infty})})$) on $T^{\Xi}(A_{\bar{s}})$ ( resp.
$V^{\Xi}(A_{\bar{s}})$).

We define a sheaf of sets $ILV^{(\Xi)}:
(Sch_{/S})_{\acute{e}t}\rightarrow Sets$  on the \'{e}tale site of
the category of $S$-schemes such that for any connected $S$-scheme
$S'$, we have:
$$ILV^{(\Xi)}(S')=H^0(\pi(S',\bar{s}'),\ \mathrm{Isom}_{\mathcal{O}_F}(L_\mathfrak{c}\otimes_{\mathcal{O}_F}\widehat{\mathbb{Z}}^{(\Xi)},T^{\Xi}(A_{\bar{s}'}))),$$
where $\bar{s}'$ is a geometric point of $S'$ over a point $s'$ of
$S'$. The \'{e}tale sheaf $ILV^{(\Xi)}$ is independent of the choice
of $s'$ (see \cite{H04} Section 6.4.1).  The group
$G(\widehat{\mathbb{Z}}^{(\Xi)})$ acts on the sheaf $ILV^{(\Xi)}$
through its action on the Tate module $T^{\Xi}(A_{\bar{s}'})$, and
we denote by $ILV^{(\Xi)}/K$ the quotient sheaf of $ILV^{(\Xi)}$
under the group action of $K^{(\Xi)}$. An $K$-integral level
structure of $A_{/S}$ is a section $\bar{\alpha}^{(\Xi)}\in
ILV^{(\Xi)}/K(S)$. Similarly we define another sheaf $RLV^{(\Xi)}:
(Sch_{/S})_{\acute{e}t}\rightarrow Sets$ such that  for any
connected $S$-scheme $S'$, we have:
$$RLV^{(\Xi)}(S')=H^0(\pi(S',\bar{s}'),\mathrm{Isom}_{\mathcal{O}_F}(V\otimes_{\mathbb{Z}}\mathbb{A}^{(\Xi \infty)},V^{\Xi}(A_{\bar{s}'}))),$$
and define the quotient sheaf $RLV^{(\Xi)}/K$ in the same way. Then
a $K$-rational level structure of $A_{/S}$ is a section
$\bar{\eta}^{(\Xi)}\in RLV^{(\Xi)}/K(S)$.

Suppose that we have another abelian scheme $A'_{/S}$ and a
homomorphism $\iota':\mathcal{O}_F\rightarrow\mathrm{End}(A'_{/S})$.
We can similarly define two \'{e}tale sheaves $ILV'^{(\Xi)}$ and
$RLV'^{(\Xi)}$ replacing $A_{/S}$ by $A'_{/S}$ in the above
construction. If $f:A\rightarrow A'$ is an $\mathcal{O}_F$-linear
isomorphism of abelian schemes, the isomorphism $f$ induces an
isomorphism of Tate modules $T^{(\Xi)}(A_{\bar{s}})\cong
T^{(\Xi)}(A'_{\bar{s}})$ for any geometric point $\bar{s}$ of $S$.
Hence $f$ induces an isomorphism of \'{e}tale sheaves $T^{(\Xi)}(f):
ILV^{(\Xi)}\rightarrow ILV'^{(\Xi)}$ which is compatible with the
$G(\widehat{\mathbb{Z}}^{(\Xi)})$-action. Thus $f$ also induces an
isomorphism $T^{(\Xi)}(f): ILV^{(\Xi)}/K\rightarrow ILV'^{(\Xi)}/K$
for all subgroup $K$ of $G(\widehat{\mathbb{Z}})$. For any
$K$-integral level structure $\bar{\alpha}^{(\Xi)}\in
ILV^{(\Xi)}/K(S)$, we use $T^{(\Xi)}(f)(\bar{\alpha}^{(\Xi)})$ to
denote its image under the isomorphism $T^{(\Xi)}(f)$. Similarly if
$f:A\rightarrow A'$ is an $\mathcal{O}_F$-linear prime-to-$\Xi$
isogeny of abelian schemes, then $f$ induces an isomorphism
$V^{(\Xi)}(A_{\bar{s}})\cong V^{(\Xi)}(A'_{\bar{s}})$ and hence
isomorphisms of \'{e}tale sheaves $V^{(\Xi)}(f):
RLV^{(\Xi)}\rightarrow RLV'^{(\Xi)}$ and $V^{(\Xi)}(f):
RLV^{(\Xi)}/K\rightarrow RLV'^{(\Xi)}/K$. For any $K$-rational level
structure $\bar{\eta}^{(\Xi)}\in RLV^{(\Xi)}/K(S)$, we use
$V^{(\Xi)}(f)(\bar{\eta}^{(\Xi)})$ to denote its image under the
isomorphism $V^{(\Xi)}(f)$. We refer to \cite{H06} Section $4.3.1$
for more discussion on this topic.
\end{remark}

\begin{theorem}\label{Th1}
When $K$ is small enough (e.g. $\det(K^{(\Xi)})\cap
\mathcal{O}_{+}^\times\subseteq (K^{(\Xi)}\cap Z(\mathbb{Z}))^2$),
then  we have a natural isomorphism of functors:

$$i:\coprod_{\mathfrak{c}\in
I}\mathscr{E}_{K,\mathfrak{c}}^{(\Xi)}\rightarrow\mathscr{E'}_K^{(\Xi)}.$$

\end{theorem}

The proof is essentially given in \cite{H04} Section $4.2.1$ so we
omit the proof here. The only thing we want to remark here is that
for any quadruple $(A_{/S},\iota,\bar{\lambda},\bar{\eta}^{(\Xi)})$
considered in Definition \ref{def3.1}, we can find an abelian scheme
$A'_{/S}$ with real multiplication $\iota'$,and  an
$\mathcal{O}_F$-linear prime-to-$\Xi$ isogeny $f:A\rightarrow A'$ of
abelian schemes over $S$ such that $A'_{/S}$ admits an integral
level structure. Since $S$ is a $\mathbb{Z}_{(\Xi)}$-scheme, the
isogeny $f$ is \'{e}tale. From Lemma \ref{2.10}, the pair
$(A'_{/S},\iota')$ also satisfies the condition $(\mathrm{DP})$.
Then we can follow the argument in \cite{H04} Section $4.2.1$ to
conclude this theorem.

From Theorem \ref{Th1}, when $K$ is small enough, the functor
$\mathscr{E}_{K,\mathfrak{c}}^{(\Xi)}$ is representable. By Theorem
\ref{Th1}, we can assume that the functor $\mathscr{E'}_K^{(\Xi)}$
is represented by a $\mathbb{Z}_{(\Xi)}$-scheme $Sh_K^{(\Xi)}$. From
\cite{DP} Theorem $2.2$, the scheme $Sh_K^{(\Xi)}$ is flat of
complete intersection over $\mathbb{Z}_{(\Xi)}$,and smooth over
$\mathbb{Z}_{(\Xi)}[\frac{1}{d_F}]$.

Now we take the projective limit of $Sh_K^{(\Xi)}$ for various $K$,
and get a $\mathbb{Z}_{(\Xi)}$-scheme $Sh^{(\Xi)}$. It is clear that
$Sh^{(\Xi)}_{/\mathbb{Z}_{(\Xi)}}$ represents the moduli problem
$\mathscr{E'}^{(\Xi)}: Sch_{/\mathbb{Z}_{(\Xi)}}\rightarrow Set$,
such that for each $\mathbb{Z}_{(\Xi)}$-scheme $S$,
$\mathscr{E'}_K^{(\Xi)}(S)=[(A_{/S},\iota,\bar{\lambda},\eta^{(\Xi)})]$.
where $[(A_{/S},\iota,\bar{\lambda},\bar{\eta}^{(\Xi)})]$ is the set
of isomorphism classes of quadruples
$(A_{/S},\iota,\bar{\lambda},\eta^{(\Xi)})$ considered in Definition
\ref{def3.1}, except that $\eta^{(\Xi)}\in RLA^{(\Xi)}(S)$ is a
rational level structure instead of  a $K$-rational level structure
for some open compact subgroup $K$. An isomorphism from one
quadruple $(A_{/S},\iota,\bar{\lambda},\eta^{(\Xi)})$ to another
$(A'_{/S},\iota',\bar{\lambda}',\eta'^{(\Xi)})$ is an element $f\in
\mathrm{Hom}(A_{/S},A'_{/S})\otimes_\mathbb{Z}\mathbb{Z}_{(\Xi)}$
whose degree is prime to $\Xi$ such that it satisfies the first two
conditions in Definition \ref{def3.1}, and also $V^{(\Xi)}(f)(
\eta^{(\Xi)})=\eta'^{(\Xi)}$ instead of that last condition there.

For any $g\in G(F_{\mathbb{A}^{(\Xi\infty)}})$, the map sending each
quadruple $(A_{/S},\iota,\bar{\lambda},\eta^{(\Xi)})$ to another
quadruple
\\$(A_{/S},\iota,\bar{\lambda},g(\eta^{(\Xi)}))$ induces an
automorphism of the functor $\mathscr{E'}^{(\Xi)}$, and hence an
automorphism of the Shimura variety
$Sh^{(\Xi)}_{/\mathbb{Z}_{(\Xi)}}$ by universality. We still denote
this action by $g$.

For simplicity we denote the Shimuar variety
$Sh^{(\Xi)}_{/\mathbb{Z}_{(\Xi)}}$ by $X_{/\mathbb{Z}_{(\Xi)}}$ in
the following discussion. Pick a closed point $x_p\in
X(\bar{\mathbb{F}}_p)$. Let $K$ be a neat subgroup of
$G(F_{\mathbb{A}^{(\Xi\infty)}})$. Then the natural morphism
$X\rightarrow X_K=X/K$ is \'{e}tale. Let $\mathcal{O}_{X,x_p}$ and
$\mathcal{O}_{X_K,x_p}$ be the stalk of $X$ and $X_K$ at $x_p$,
respectively. The completion of $\mathcal{O}_{X,x_p}$ is canonically
isomorphic to the completion of $\mathcal{O}_{X_K,x_p}$, and we
denote this completion by $\widehat{\mathcal{O}}_{x_p}$. Suppose
that $x_p$ is represented by a quadruple
$(A_{0/\bar{\mathbb{F}}_p},\iota_0,\phi_0,\bar{\alpha_0}^{(\Xi)})\in
\mathscr{E}_{K,\mathfrak{c}}^{(\Xi)}(\bar{\mathbb{F}}_p)$.

Let $CL_{/W_p}$ be the category of complete local $W_p$-algebras
with residue field $\bar{\mathbb{F}}_p$. Consider the local
deformation functor $\widehat{D}_p:CL_{/W_p}\rightarrow Set$, given by
$$\widehat{D}_p(R)=\{(A_{/R},\iota_R,\phi_R)|(A_{/R},\iota_R,\phi_R)\times_R
\bar{\mathbb{F}}_p\cong
(A_{0/\bar{\mathbb{F}}_p},\iota_0,\phi_0)\}_{/\cong},$$
 here the
triple $(A_{/R},\iota_R,\phi_R)$ consists of an abelian $A$ schemes
over $R$, an algebra homomorphism
$\iota_R:\mathcal{O}_F\rightarrow\mathrm{End}(A_{/R})$ and a
$\mathfrak{c}$-polarization $\phi_R$ of $A_{/R}$. An isomorphism
from a triple $(A_{/R},\iota_R,\phi_R)$ to another
$(A'_{/R},\iota'_R,\phi'_R)$ is an isomorphism $f:A\rightarrow A'$
of abelian schemes over $R$ such that
\begin{enumerate}
\item for all $a\in \mathcal{O}_F$, we have $f\circ \iota_R(a)=\iota'_R(a)\circ f:A\rightarrow A'$;
\item $f^t\circ \phi'_R\circ (f\otimes Id_\mathfrak{c})=\phi_R:A\otimes_{\mathcal{O}_F}\mathfrak{c}\rightarrow A^t$.
\end{enumerate}

Define a functor $DEF_p:CL_{/W_p}\rightarrow Set$ by the formula:
$$DEF_p(R)=\{(D_{/R},\Lambda_R,\varepsilon_R)\}_{/\cong},$$
where $D_{/R}$ is a Barsotti-Tate $\mathcal{O}_F$-module over $R$,
$\Lambda_R: D\otimes_{\mathcal{O}_F}\mathfrak{c}\rightarrow D^t$ is
an $\mathcal{O}_F$-linear  isomorphism of Barsotti-Tate
$\mathcal{O}_F$-modules over $R$ ($D^t$ is the Cartier dual of $D$),
and $\varepsilon_R:D_0=D\otimes_R \bar{\mathbb{F}}_p\rightarrow
A_0[p^\infty]$ is an isomorphism of Barsotti-Tate
$\mathcal{O}_F$-modules over the special fiber
$\mathrm{Spec}(\bar{\mathbb{F}}_p)$ of $\mathrm{Spec}(R)$.

For any triple $(A_{/R},\iota_R,\phi_R)$ in $\widehat{D}_p(R)$, let
$A[p^\infty]_{/R}$ be its $p$-divisible Barsotti-Tate
$\mathcal{O}_F$-module over $R$. The $\mathfrak{c}$-polarization
$\phi_R$ of $A_{/R}$ gives an isomorphism $\Lambda_R:
A[p^\infty]\otimes_{\mathcal{O}_F}\mathfrak{c}\rightarrow
A^t[p^\infty]\cong ( A[p^\infty])^t$. The isomorphism
$(A_{/R},\iota_R,\phi_R)\times_R \bar{\mathbb{F}}_p\cong
(A_{0/\bar{\mathbb{F}}_p},\iota_0,\phi_0)$ gives an isomorphism
$\varepsilon_R: A[p^\infty]\otimes_R \bar{\mathbb{F}}_p\rightarrow
A_0[p^\infty]$. By the Serre-Tate deformation theory (\cite{K0}
Theorem $1.2.1$), we have:

\begin{proposition}\label{3.3}
The above association $(A_{/R},\iota_R,\phi_R)\mapsto (A[p^\infty]_{/R},\Lambda_R,\varepsilon_R)$ induces an equivalence of functors $\widehat{D}_p\rightarrow DEF_p$.
\end{proposition}

We  define two more
functors $DEF_p^?:CL_{/W_p}\rightarrow Set$,$?=ord,ll$, by:
$$DEF_p^?(R)=\{(D^?,\phi^?,\varepsilon^?)\}_{/\cong},$$
here in the triple $( D^?,\phi^?,\varepsilon^?)$, $D^?$ is a
Barsotti-Tate $\mathcal{O}_F$-module over
$R$,$\phi^?:D^?\otimes_{\mathcal{O}_F}\mathfrak{c}\rightarrow(D^?)^t$
is an isomorphism of Barsotti-Tate $\mathcal{O}_F$-modules over $R$,
and $\varepsilon^?: D^?\otimes_R \bar{\mathbb{F}}_p\rightarrow
A_0[p^\infty]^?$ is an isomorphism over $\bar{\mathbb{F}}_p$.

Similar with \cite{H12} Proposition $1.2$, we have the following facts:
\begin{enumerate}
\item the functor $DEF_p$ is represented by the formal scheme
$\widehat{S}_{p/W_p}$ associated to $\widehat{\mathcal{O}}_{x_p}$;
\item there is a natural equivalence of functors: $DEF_p\cong DEF_p^{ord}\times
DEF_p^{ll}$,and hence the formal $\widehat{S}_{p/W_p}$ is a product of
two formal schemes $\widehat{S}^{ord}_{p/W_p}$ and
$\widehat{S}^{ll}_{p/W_p}$ such that $DEF_p^?$ is represented by
$\widehat{S}^{?}_{p/W_p}$ for $?=ord, ll$;
\item For each $\mathfrak{p}\in \Sigma_p^{ord}$, fix an isomorphism $\mathcal{O}_{\mathfrak{p}}\cong
T_{\mathfrak{p}}(A_0)$ (recall that $T_{\mathfrak{p}}(A_0)$ is the $\mathfrak{p}$-adic Tate module of $A_0$ defined in Section 1).
Since $\mathfrak{c}$ is prime to $p$, by the
$\mathfrak{c}$-polarization $\phi_0$, we also have an isomorphism
$\mathcal{O}_{\mathfrak{p}}\cong T_{\mathfrak{p}}(A^t_0)$. Then
$\widehat{S}^{ord}_{p/W_p}$ is a smooth formal scheme over $W_p$
which is isomorphic to
$$\prod_{\mathfrak{p}\in
\Sigma^{ord}}\mathrm{Hom}(T_{\mathfrak{p}}(A_0)\otimes_{\mathcal{O}_{\mathfrak{p}}}T_{\mathfrak{p}}(A^t_0),
\widehat{\mathbb{G}}_m)\cong \prod_{\mathfrak{p}\in
\Sigma^{ord}}\mathrm{Hom}(\mathcal{O}_{\mathfrak{p}},\widehat{\mathbb{G}}_m)=\prod_{\mathfrak{p}\in
\Sigma^{ord}}\widehat{\mathbb{G}}_m\otimes_{\mathbb{Z}_p}\mathcal{O}_{\mathfrak{p}}^\ast,$$
here
$\mathcal{O}_{\mathfrak{p}}^\ast=\mathrm{Hom}_{\mathbb{Z}_p}(\mathcal{O}_{\mathfrak{p}},\mathbb{Z}_p)$.
\end{enumerate}

In fact, for any triple $(A_{/R},\iota_R,\phi_R)$ in $\widehat{D}_p(R)$, the level structure $\bar{\alpha_0}^{(\Xi)}$ on $A_0$
can be extended uniquely to a level structure on $A_{/R}$. Then the functor $\widehat{D}_p$, and hence the functor $DEF_p$ by Proposition
\ref{3.3}, is represented by the the formal scheme $\widehat{S}_{p/W_p}=Spf(\widehat{\mathcal{O}}_{x_p})$.

For a triple $(D_{/R},\Lambda_R,\varepsilon_R)\in DEF_p(R)$, we have a canonical decomposition  the Barsotti-Tate $\mathcal{O}_F$-module
$D=D^{ord}\times D^{ll}$, where $D^{ord}$ is the maximal ordinary Barsotti-Tate $\mathcal{O}_F$-submodule of $D$, and $D^{ll}$ is its local-local
complement. From this we have a morphism
$$DEF_p(R)\ni (D_{/R},\Lambda_R,\varepsilon_R)\mapsto \{(D^{ord}_{/R},\Lambda_R|_{D^{ord}},\varepsilon_R|_{D^{ord}}),(D^{ll}_{/R},\Lambda_R|_{D^{ll}},\varepsilon_R|_{D^{ll}})\}\in DEF^{ord}_p(R)\times DEF^{ll}_p(R),$$
from which we get a equivalence of functors between $DEF_p$ and $DEF_p^{ord}\times DEF_p^{ll}$. Hence the formal scheme $\widehat{S}_{p/W_p}$
is a product of two formal schemes $\widehat{S}^{ord}_{p/W_p}\times\widehat{S}^{ll}_{p/W_p}$.

In contrast with \cite{H12} Proposition $1.2$, the formal scheme
$\widehat{S}_{p/W_p}$ may not be smooth when $p$ divides the
discriminant $d_F$ of $F$ since the Shimura variety
$Sh^{(p)}_{/\mathbb{Z}_{(p)}}$ we consider here is not smooth. But
from the Serre-Tate deformation theory, the formal scheme
$\widehat{S}^{ord}_{p/W_p}$ is always smooth, and this is the  part
we are  interested in.

\section{Eigen coordinates}

Let $k\subseteq \bar{\mathbb{Q}}$ be a number field and $\Xi$ be a finite set of primes. For each $p\in \Xi$, choose a
finite extension $\widetilde{K}_p$ of $K_p$ in $\bar{\mathbb{Q}}_p$
such that:

\begin{enumerate}
\item $k\subseteq i_p^{-1}(\widetilde{K}_p)$;
\item $i_p^{-1}(\widetilde{K}_p)$ contains the Galois closure of $F$ in
$\bar{\mathbb{Q}}$.
\end{enumerate}
Denote by $\widetilde{W}_p$ the integer ring of $\widetilde{K}_p$. Then
define:
$$\widetilde{\mathcal{W}}_\Xi=\bigcap_{p\in \Xi} i_p^{-1}(\widetilde{W}_p)\subseteq\bar{\mathbb{Q}},\widetilde{\mathcal{W}}_k=
\widetilde{\mathcal{W}}_\Xi \cap k.$$ The ring
$\widetilde{\mathcal{W}}_\Xi$ is a semilocal ring, and for each
$l\in \Xi$, there is a unique maximal ideal $\mathfrak{m}_l$ with
residue characteristic $l$. Let $\widetilde{\mathcal{K}}_\Xi$ be the
quotient field of $\widetilde{\mathcal{W}}_\Xi$.

Suppose that the quadruple
$(A_{/\widetilde{\mathcal{W}}_\Xi},\iota,\bar{\lambda},\eta^{(\Xi)})$
represents a point $x\in Sh^{(\Xi)}(\widetilde{\mathcal{W}}_\Xi)$
such that the image of $x$ lies in
$Sh^{(\Xi)}(\widetilde{\mathcal{W}}_k)$. For each $p\in \Xi$, $x$
induces an $\bar{\mathbb{F}}_p$-valued point $x_p\in
Sh^{(\Xi)}(\bar{\mathbb{F}}_p)$. Then the quadruple the quadruple
$(A_{\mathfrak{P}/\bar{\mathbb{F}}_p},\iota_\mathfrak{P},\bar{\lambda}_\mathfrak{P},\eta_\mathfrak{P}^{(\Xi)})$
obtained by mod $p$ reduction represents the point $x_p$.

By \cite{H12} Lemma $2.2$, we have

\begin{lemme}\label{4.1}
If $A_{\mathfrak{P}/\bar{\mathbb{F}}_p}$ is not supersingular (i.e.
$\Sigma_p^{ord}\neq\emptyset$), then there exists a CM quadratic
extension $M$ of $F$, and an isomorphism of $F$-algebras
$\theta_{\mathfrak{P}}:M\cong$
$\mathrm{End}^0_F(A_{\mathfrak{P}/\bar{\mathbb{F}}_p})$. Set
$R=M\cap\theta_{\mathfrak{P}}^{-1}(\mathrm{End}_{\mathcal{O}_F}(A_{\mathfrak{P}/\bar{\mathbb{F}}_p}))$,
which is an order in $M$. If a prime ideal $\mathfrak{p}$ in
$\mathcal{O}_F$ belongs to $\Sigma^{ord}_p$; i.e.
$A_{\mathfrak{P}}[\mathfrak{p}]$ has nontrivial
$\bar{\mathbb{F}}_p$-rational points, then $\mathfrak{p}$ splits
into two primes $\mathcal{P}\bar{\mathcal{P}}$ in $R$ with
$\mathcal{P}\neq\bar{\mathcal{P}}$.
\end{lemme}
As in \cite{H12}, we make the convention that we choose
$\mathcal{P}$ such that $A_{\mathfrak{P}}[\mathcal{P}]$ is connected
and $A_{\mathfrak{P}}[\bar{\mathcal{P}}]$ is \'{e}tale.

By the above lemma, we have an isomorphism $M\otimes_F
F_{\mathfrak{p}}\cong F_{\mathfrak{p}}\times F_{\mathfrak{p}}$, such
that the first factor corresponds to $\mathcal{P}$ and the second
factor  corresponds to $\bar{\mathcal{P}}$. As $M$ can be naturally
embedded into $M\otimes_F F_{\mathfrak{p}}$, we have two embeddings
from $M$ to $F_{\mathfrak{p}}$, which correspond to the two factors
of $F_{\mathfrak{p}}\times F_{\mathfrak{p}}$. We always regard $M$
as a subfield of $F_{\mathfrak{p}}$ by the first embedding, while
the second embedding is denoted by $c:M\hookrightarrow
F_{\mathfrak{p}}$.

 Let $R_{(\Xi)}=R\otimes_{\mathbb{Z}}\mathbb{Z}_{(\Xi)}$. For $\alpha\in
R_{(\Xi)}^\times$, $\theta_{\mathfrak{P}}(\alpha)$ is a
prime-to-$\Xi$ isogeny of $A_{\mathfrak{P}/\bar{\mathbb{F}}_p}$, and
hence induces an endomorphism of $V^{(\Xi)}(A_{\mathfrak{P}})$. We
still denote this endomorphism by $\theta_{\mathfrak{P}}(\alpha)$.
Define a map $\hat{\rho}:R_{(\Xi)}^\times\rightarrow
G(F_{\mathbb{A}^{(\Xi\infty)}})$ such that for each $\alpha\in
R_{(\Xi)}^\times$, $\hat{\rho}(\alpha)$ is given by the formula:
$\eta_{\mathfrak{P}}^{(\Xi)}\circ\hat{\rho}(\alpha)=\theta_{\mathfrak{P}}(\alpha)\circ\eta_{\mathfrak{P}}^{(\Xi)}$.

Fix a prime-to-$\Xi$ polarization $\lambda_{\mathfrak{P}}$ of
$A_{\mathfrak{P}}$ as a representative of
$\bar{\lambda}_{\mathfrak{P}}$. Under the isomorphism
$\theta_{\mathfrak{P}}$, the Rosati involution associated to
$\lambda_{\mathfrak{P}}$ on
$\mathrm{End}^0_F(A_{\mathfrak{P}/\bar{\mathbb{F}}_p})$ induces a
positive involution on filed $M$. As $M$ is CM, this involution must
be the complex conjugation on $M$. Hence for any $\alpha\in
R_{(\Xi)}^\times$,
$\lambda_{\mathfrak{P}}^{-1}\circ\theta_{\mathfrak{P}}(\alpha)^t\circ\lambda_{\mathfrak{P}}=\theta_{\mathfrak{P}}(\bar{\alpha})$.
Then
$\theta_{\mathfrak{P}}(\alpha)^t\circ\lambda_{\mathfrak{P}}\circ\theta_{\mathfrak{P}}(\alpha)=\lambda_{\mathfrak{P}}\circ\theta_{\mathfrak{P}}
(\bar{\alpha})\circ\theta_{\mathfrak{P}}(\alpha)=\lambda_{\mathfrak{P}}\circ\theta_{\mathfrak{P}}(\alpha\bar{\alpha})$.
Since $\alpha\bar{\alpha}\in \mathcal{O}_{(\Xi),+}^\times$,  we have
$\theta_{\mathfrak{P}}(\alpha)^t\circ\bar{\lambda}_{\mathfrak{P}}\circ\theta_{\mathfrak{P}}(\alpha)=\bar{\lambda}_{\mathfrak{P}}$.
So $\theta_{\mathfrak{P}}(\alpha)$ is an isogeny from the quadruple
$(A_{\mathfrak{P}/\bar{\mathbb{F}}_p},\iota_\mathfrak{P},\bar{\lambda}_\mathfrak{P},\eta_\mathfrak{P}^{(\Xi)})$
to
$(A_{\mathfrak{P}/\bar{\mathbb{F}}_p},\iota_\mathfrak{P},\bar{\lambda}_\mathfrak{P},\theta_{\mathfrak{P}}(\alpha)(\eta_\mathfrak{P}^{(\Xi)})
)=(A_{\mathfrak{P}/\bar{\mathbb{F}}_p},\iota_\mathfrak{P},\bar{\lambda}_\mathfrak{P},\hat{\rho}(\alpha)(\eta_\mathfrak{P}^{(\Xi)}))$
in the sense of Definition \ref{def3.1}; in other words, the
automorphism $g=\hat{\rho}(\alpha)$ of the Shimura variety
$Sh^{(\Xi)}_{/\mathcal{W}_\Xi}=Sh^{(\Xi)}_{/\mathbb{Z}_{(\Xi)}}\times_{\mathbb{Z}_{(\Xi)}}\mathcal{W}_\Xi$
fixes the closed point $x_p$.

Denote the formal scheme $\widehat{S}_{p/W_p}$ as the completion of
the Shimura variety $Sh^{(\Xi)}_{/\mathcal{W}_\Xi}$ along the closed
point $x_p$, and $\nu_p:\widehat{S}_{p/W_p}\rightarrow
Sh^{(\Xi)}_{/W_p}$ is the natural morphism. As explained in Section
$3$, $\widehat{S}_{p/W_p}$ is the product of two formal schemes
$\widehat{S}^{ord}_{p/W_p}$ and $\widehat{S}^{ll}_{p/W_p}$,
 and if we fix
an isomorphism  $\mathcal{O}_{\mathfrak{p}}\cong
T_{\mathfrak{p}}(A_\mathfrak{P})$ for each $\mathfrak{p}\in
\Sigma^{ord}$, then $\widehat{S}^{ord}_{p/W_p}$ is isomorphic to
$\Pi_{\mathfrak{p}\in
\Sigma^{ord}}\widehat{\mathbb{G}}_m\otimes_{\mathbb{Z}_p}\mathcal{O}_{\mathfrak{p}}^\ast$.
By deformation theory, we have a Serre-Tate coordinate
$t_{\mathfrak{p}}\in
\widehat{\mathbb{G}}_m\otimes_{\mathbb{Z}_p}\mathcal{O}_{\mathfrak{p}}^\ast$
for each $\mathfrak{p}\in \Sigma^{ord}$. Then for each object
$\mathcal{R}$ in the category $CL_{/W_p}$, and an
$\mathcal{R}$-valued point $x\in \widehat{S}_{p}(\mathcal{R})$, the
Serre-Tate coordinate gives us an element $t_{\mathfrak{p}}(x)\in
\widehat{\mathbb{G}}_m(\mathcal{R})
\otimes_{\mathbb{Z}_p}\mathcal{O}_{\mathfrak{p}}^\ast=(1+\mathfrak{m}_{\mathcal{R}})\otimes_{\mathbb{Z}_p}\mathcal{O}_{\mathfrak{p}}^\ast$,
here $\mathfrak{m}_{\mathcal{R}}$ is the maximal ideal of
$\mathcal{R}$. In particular, when $\mathcal{R}$ is a subring of
$\mathbb{C}_p$, we can consider the $p$-adic logarithm
$log_p:\mathcal{R}\rightarrow\mathbb{C}_p$. Consider the following
map:
$$log_p\otimes Id:(1+\mathfrak{m}_{\mathcal{R}})\otimes_{\mathbb{Z}_p}\mathcal{O}_{\mathfrak{p}}^\ast\rightarrow
\mathbb{C}_p\otimes_{\mathbb{Z}_p}\mathcal{O}_{\mathfrak{p}}^\ast\cong
\mathrm{Hom}(\mathcal{O}_{\mathfrak{p}},\mathbb{C}_p)
\cong\prod_{\sigma:F\rightarrow\bar{\mathbb{Q}},\sigma\thicksim\mathfrak{p}}\mathbb{C}_p.$$
Here the notation $\sigma\thicksim\mathfrak{p}$ means that the
composite map $i_p\circ \sigma: F\rightarrow\bar{\mathbb{Q}}_p$
induces the prime $\mathfrak{p}$ of $F$. For such $\sigma$, let
$\pi_\sigma$ be the projection of
$\Pi_{\sigma:F\rightarrow\bar{\mathbb{Q}},\sigma\thicksim\mathfrak{p}}\mathbb{C}_p$
to its $\sigma$-factor. Then we get an element
$\tau_\sigma(x)=\pi_\sigma\circ(log_p\otimes
Id)(t_{\mathfrak{p}}(x))\in \mathbb{C}_p$. The association $x\in
\widehat{S}_{p}(\mathcal{R})\mapsto \tau_\sigma(x)\in\mathbb{C}_p$
gives $p$-adic rigid analytic functions on the rigid analytic space
$(\widehat{S}^{ord}_{p})^{p-an}$ associated to
$\widehat{S}^{ord}_{p}$.

Since the action of  $g=\hat{\rho}(\alpha)$ on the Shimura variety
$Sh^{(\Xi)}_{/\mathcal{W}_\Xi}$ fixes the closed point $x_p$, this
action also preserves the formal schemes $\widehat{S}^{ord}_{p}$ and
$\widehat{S}^{ll}_{p}$, and hence $g=\hat{\rho}(\alpha)$ acts on the
function $\tau_\sigma$ for each $\sigma\thicksim\mathfrak{p}$,
$\mathfrak{p}\in \Sigma^{ord}$. By \cite{H10} Lemma $3.3$, the
action of $g=\hat{\rho}(\alpha)$ on the Serre-Tate coordinate
$t_\mathfrak{p}$ is given by the formula
$g(t_\mathfrak{p})=t_\mathfrak{p}^{\alpha^{1-c}}$. (See the
explanation after Lemma \ref{4.1} for the two embeddings of $M$ to
$F_\mathfrak{p}$). Then by the construction of $\tau_\sigma$, we see
that the action of $g=\hat{\rho}(\alpha)$ on the function
$\tau_\sigma$ is given by the formula:
$g(\tau_\sigma)=\tau_\sigma\circ \hat{\rho}(\alpha)=i_p\circ
\sigma(\alpha^{1-c})\cdot\tau_\sigma$. We remark here that $i_p\circ
\sigma:F\rightarrow \bar{\mathbb{Q}}_p$ naturally extends to an
embedding $i_p\circ \sigma:F_\mathfrak{p}\rightarrow
\bar{\mathbb{Q}}_p$, and hence the expression $i_p\circ
\sigma(\alpha^{1-c})$ is well defined. As in \cite{H12}, the
function $\tau_\sigma$ is called a $\hat{\rho}$-eigen
$\sigma$-coordinate.

Now consider the original point $x\in
Sh^{(\Xi)}(\widetilde{\mathcal{W}}_\Xi)$, which is represented by
the quadruple
$(A_{/\widetilde{\mathcal{W}}_\Xi},\iota_,\bar{\lambda},\eta^{(\Xi)})$.
Assume that we have a prime $\mathfrak{p}\in \Sigma^{ord}$, such
that the exact sequence of Barsotti-Tate
$\mathcal{O}_{\mathfrak{p}}$-modules :
$$0\rightarrow\mu_{p^\infty}\otimes_{\mathbb{Z}_p}\mathcal{O}_{\mathfrak{p}}^\ast \rightarrow A[\mathfrak{p}^\infty]\rightarrow F_{\mathfrak{p}}/\mathcal{O}_{\mathfrak{p}}\rightarrow 0$$
splits over over $\widetilde{W}_p$. In this case, the Serre-Tate
coordinate $t_{\mathfrak{p}}(x)$ for the prime $\mathfrak{p}$ at the
point $x$ must be a p-th power root of unity (see \cite{Br} Section
$7$ or \cite{ECAI} Section $5.3.4$ for a proof of this fact,
although only the coordinate of an elliptic curve is discussed
there, the same argument works in the higher dimension as well).
Hence for the $\hat{\rho}$-eigen coordinate
 we have $\tau_\sigma(x)=1$ for
all $\sigma\sim\mathfrak{p}$.

We can also regard $x\in Sh^{(\Xi)}(\widetilde{\mathcal{W}}_\Xi)$ as a $\widetilde{W}_p$-rational point. Then $x$ actually sits in the formal scheme
 $\widehat{S}_{p/W_p}$,
in other words, if we regard $x$ as a morphism
$\mathrm{Spec}(\widetilde{\mathcal{W}}_\Xi)\rightarrow Sh^{(\Xi)}$,
then this morphism factors through $\nu_p:\widehat{S}_p\rightarrow
Sh^{(\Xi)}$.

Let $(A_p^{univ},\iota_p^{univ},\phi_p^{univ})$ be the universal
object over $\widehat{S}_p$. Let $\pi_p: A_p^{univ}\rightarrow
\widehat{S}_p$ be the structure morphism and
$e_p:\widehat{S}_p\rightarrow A_p^{univ}$ be the morphism
corresponding to the identity element. Consider the sheaf
$\omega_p^{univ}=(\pi_{p})_\ast(\Omega_{A_p^{univ}/\widehat{S}_p})=e_p^\ast(\Omega_{A_p^{univ}/\widehat{S}_p})$
over $\widehat{S}_{p/W_p}$, which has a natural
$\mathcal{O}_{\widehat{S}_p}\otimes_{\mathbb{Z}}\mathcal{O}_F$-module
structure, and compatible with arbitrary base change. Set
$(\omega_p^{univ})^{\otimes 2}=
\omega_p^{univ}\otimes_{(\mathcal{O}_{\widehat{S}_p}\otimes_{\mathbb{Z}}\mathcal{O}_F)}\omega_p^{univ}$.
Then we have the Kodaira-Spencer map:
$$KS:(\omega_p^{univ})^{\otimes 2}\rightarrow \Omega_{\widehat{S}_p/W_p}.$$
We remark here that the Kodaira-Spencer map is $\mathcal{O}_{\widehat{S}_p}\otimes_{\mathbb{Z}}\mathcal{O}_F$-linear and compatible with the
$g=\hat{\rho}(\alpha)$-action on both sides.

By the isomorphism $\widehat{S}_p\cong \widehat{S}^{ord}_p\times
\widehat{S}^{ll}_p$ over $W_p$, we have the decomposition:
$\Omega_{\widehat{S}_p/W_p}=(\pi^{ord})^\ast\Omega_{\widehat{S}^{ord}_p/W_p}\oplus(\pi^{ll})^\ast\Omega_{\widehat{S}^{ll}_p/W_p}$,
where $\pi^{ord}:\widehat{S}_p\rightarrow \widehat{S}^{ord}_p$ and
$\pi^{ll}:\widehat{S}_p\rightarrow \widehat{S}^{ll}_p$ are the
natural projection. Since $\widehat{S}^{ord}_p\cong
\Pi_{\mathfrak{p}\in\Sigma^{ord}}\widehat{\mathbb{G}}_m\otimes_{\mathbb{Z}_p}\mathcal{O}_\mathfrak{p}^\ast$,
if we set
$\widehat{S}_\mathfrak{p}=\widehat{\mathbb{G}}_m\otimes_{\mathbb{Z}_p}\mathcal{O}_\mathfrak{p}^\ast$,
then we have
$\Omega_{\widehat{S}^{ord}_p/W_p}=\oplus_{\mathfrak{p}\in\Sigma^{ord}}(\pi_\mathfrak{p})^\ast\Omega_{\widehat{S}_\mathfrak{p}/W_p}$,
where $\pi_\mathfrak{p}: \widehat{S}^{ord}_p\rightarrow
\widehat{S}_\mathfrak{p}$ is the natural projection. To express the
$g$-action on $\Omega_{\widehat{S}^{ord}_p/W_p}$ in a simple way, we
base change this  module to $\widetilde{K}_p$, i.e. we consider
$\Omega_{\widehat{S}^{ord}_p/W_p}\otimes_{W_p}\widetilde{K}_p=\Omega_{\widehat{S}^{ord}_p/\widetilde{K}_p}
=\oplus_{\mathfrak{p}\in\Sigma^{ord}}\Omega_{\widehat{S}_\mathfrak{p}/\widetilde{K}_p}$,
which is free of finite rank over
$(\widehat{S}^{ord}_p)_{/\widetilde{K}_p}$. Moreover, for each
$\mathfrak{p}\in \Sigma^{ord}$, the set $\{d\tau_\sigma|\tau\sim
\mathfrak{p}\}$ forms a basis of the module
$\Omega_{\widehat{S}_\mathfrak{p}/\widetilde{K}_p}$ over
$\widehat{S}_\mathfrak{p}$, here $\tau_\sigma$'s are the
$\hat{\rho}$-eigen coordinates constructed above.

On the other hand, we consider the cotangent bundle $(\omega_p^{univ})^{\otimes 2}\otimes_{W_p}\widetilde{K}_p=(\widetilde{\omega}_p^{univ})^{\otimes 2}$, which has a natural
$\mathcal{O}_F\otimes_\mathbb{Z}\widetilde{K}_p$-module structure. By our construction of $\widetilde{K}_p$, for any embedding
$\sigma: F\rightarrow\bar{\mathbb{Q}}$, $\sigma(\mathcal{O}_F)$ is contained in $i_p^{-1}(\widetilde{K}_p)$. Hence we have the isomorphism
$\mathcal{O}_F\otimes_\mathbb{Z}\widetilde{K}_p\cong\Pi_{\sigma: F\rightarrow\bar{\mathbb{Q}}}\widetilde{K}_p$. By this isomorphism we can
decompose the $\mathcal{O}_F\otimes_\mathbb{Z}\widetilde{K}_p$-module $(\widetilde{\omega}_p^{univ})^{\otimes 2}$ as $(\widetilde{\omega}_p^{univ})^{\otimes 2}
=\oplus_{\sigma: F\rightarrow\bar{\mathbb{Q}}}(\widetilde{\omega}_p^{univ})^{\otimes 2\sigma}$ such that on the bundle $(\widetilde{\omega}_p^{univ})^{\otimes 2\sigma}$,
$\mathcal{O}_F$ acts through the embedding $\sigma$.

Then by \cite{K1} section $1.0$, for each $\mathfrak{p}\in
\Sigma^{ord}$, the Kodaira-Spencer map induces an isomorphism
$$\bigoplus_{\sigma:F\rightarrow\bar{\mathbb{Q}},\sigma\sim \mathfrak{p}}(\widetilde{\omega}_p^{univ})^{\otimes 2\sigma}\rightarrow(\pi_\mathfrak{p}\circ\pi^{ord})^\ast\Omega_{\widehat{S}_\mathfrak{p}/\widetilde{K}_p},$$
under which the bundle $(\widetilde{\omega}_p^{univ})^{\otimes
2\sigma}$ corresponds to the sub-bundle generated by $d\tau_\sigma$.
Hence the action of $g=\hat{\rho}(\alpha)$ preserves each
$(\widetilde{\omega}_p^{univ})^{\otimes 2\sigma}$ and acts it by
multiplying the scalar $i_p\circ \sigma(\alpha^{1-c})$. Moreover, as
we assume that $\tau_\sigma(x)=0$ for all $\sigma\sim\mathfrak{p}$,
$g$ also preserves $(\widetilde{\omega}_p^{univ})^{\otimes
2\sigma}(x)$.

Now we can state the main result in this section:

\begin{theorem}\label{Th2}
Fix an embedding $\sigma_1:F\rightarrow\bar{\mathbb{Q}}$, such that
$i_p\circ\sigma_1$ induces $\mathfrak{p}$. If there exists some
prime $l\neq p$ in $\Xi$, such that the prime $\mathfrak{l}$ induced
from $i_l\circ\sigma_1$ belongs to $\Sigma_l^{ord}$, then we have an
isomorphism of $F$-algebras:
$\mathrm{End}^0_F(A_{\mathfrak{P}/\bar{F}_p})\cong\mathrm{End}^0_F(A_{\mathfrak{L}/\bar{F}_l})$.
Here $A_{\mathfrak{L}/\bar{F}_l}$ sits in the quadruple
$(A_{\mathfrak{L}/\bar{\mathbb{F}}_l},\iota_\mathfrak{L},\bar{\lambda}_\mathfrak{L},\eta_\mathfrak{L}^{(\Xi)})$
obtained by mod $l$ reduction of the point $x\in
Sh^{(\Xi)}(\widetilde{\mathcal{W}}_\Xi)$.
\end{theorem}
\begin{proof}
Set $\omega=\pi_\ast(\Omega_{A/\widetilde{\mathcal{W}}_\Xi}^1)$, which is naturally an $\mathcal{O}_F\otimes_{\mathbb{Z}}\widetilde{\mathcal{W}}_\Xi$-module.
Again we set $\omega^{\otimes 2}=\omega\otimes_{(\mathcal{O}_F\otimes_{\mathbb{Z}}\widetilde{\mathcal{W}}_\Xi)}\omega$. The base change
$\omega^{\otimes 2}\otimes_{\widetilde{\mathcal{W}}_\Xi}\widetilde{\mathcal{K}}_\Xi$ is an $\mathcal{O}_F\otimes_{\mathbb{Z}}\widetilde{\mathcal{K}}_\Xi$-module.
By our construction of $\widetilde{\mathcal{K}}_\Xi$, we have an isomorphism:
$$\mathcal{O}_F\otimes_{\mathbb{Z}}\widetilde{\mathcal{K}}_\Xi\cong\bigoplus_{\sigma:F\rightarrow\bar{\mathbb{Q}}}\widetilde{\mathcal{K}}_\Xi.$$
From this we have the decomposition: $\omega^{\otimes 2}\otimes_{\widetilde{\mathcal{W}}_\Xi}\widetilde{\mathcal{K}}_\Xi=\oplus_{\sigma:F\rightarrow\bar{\mathbb{Q}}}\tilde{\omega}^{\otimes 2\sigma}$.

Since the formation of the cotangent sheaf $\omega_p^{univ}$ over
$\widehat{S}_p$ is compatible with arbitrary base change, by the
Cartesian diagram:

$$\xymatrix{
  A \ar[d]_{} \ar[r]^{}
                & A_p^{univ} \ar[d]^{}  \\
  \mathrm{Spec}(\widetilde{K}_p)  \ar[r]^{x}
                & \widehat{S}_p, } \qquad $$ we see that $\tilde{\omega}^{\otimes
2\sigma}\otimes_{\widetilde{\mathcal{K}}_\Xi}\widetilde{K}_p=(\omega_p^{univ})^{\otimes
2\sigma}(x)$. As $g=\hat{\rho}(\alpha)$ acts on the Shimura variety
$Sh^{(\Xi)}_{/\widetilde{\mathcal{W}}_\Xi}$, $g$ sends the bundle
$\omega^{\otimes
2}\otimes_{\widetilde{\mathcal{W}}_\Xi}\widetilde{\mathcal{K}}_\Xi$
and hence each factor $\tilde{\omega}^{\otimes 2\sigma}$ to the
corresponding bundles over $g(x)$. As $g$ preserves
$(\omega_p^{univ})^{\otimes 2\sigma}(x)$ for all $\sigma\sim
\mathfrak{p}$, it also preserves $\tilde{\omega}^{\otimes 2\sigma}$.
In particular, $g$ preserves $\tilde{\omega}^{\otimes 2\sigma_1}$.

As $\tilde{\omega}^{\otimes
2\sigma_1}\otimes_{\widetilde{\mathcal{K}}_\Xi}\widetilde{K}_l=(\widetilde{\omega}_l^{univ})^{\otimes2\sigma_1}(x)$,
$g$ also preserves the fiber
$(\widetilde{\omega}_l^{univ})^{\otimes2\sigma_1}(x)$ of the bundle
$(\widetilde{\omega}_l^{univ})^{\otimes2\sigma_1}$ at the point
$x_l$ and acts on it by multiplication by
$i_l\circ\sigma_1(\alpha)$. Hence $g$ must  act on the eigen
coordinate $\tau_{\sigma_1,l}(x)$ by multiplying
$i_l\circ\sigma_1(\alpha)$, and $g$ preserves the sub-bundle of
$\Omega_{\widehat{S}_l/W_l}(x)$ generated by
$d\tau_{\sigma_1,l}(x)$. If $g$ sends $x_l\in
Sh^{(\Xi)}(\bar{\mathbb{F}}_l)$ to another point $x_l'\neq x_l$, the
action of $g$ has to move the deformation space $\widehat{S}_l$ over
$x_l$ to the deformation space $\widehat{S}'_l$ over $x'_l$, where
$\widehat{S}'_l$ is the completion of
$Sh^{(\Xi)}_{/\widetilde{\mathcal{W}}_\Xi}$  along the closed point
$x_l'$. Then $g$ induces an isomorphism of cotangent bundles
$g:\Omega_{\widehat{S}_l/W_l}(x)\rightarrow\Omega_{\widehat{S}'_l/W_l}(g(x))$
and hence $g$ cannot preserve any sub-bundle of
$\widehat{S}^{ord}_{l/W_l}(x)$, which is a contradiction.
 So $g$ fixes the point $x_l$, i.e.
there exists a prime-to-$\Xi$ isogney
$\tilde{\theta}_{\mathfrak{L}}(\alpha)$ of $A_\mathfrak{L}$,such
that $\tilde{\theta}_{\mathfrak{L}}(\alpha)\circ
\eta_\mathfrak{L}^{(\Xi)}=\eta_\mathfrak{L}^{(\Xi)}\circ
\hat{\rho}(\alpha)$, and hence establishes an isomorphism from the
quadruple
$(A_{\mathfrak{L}/\bar{\mathbb{F}}_l},\iota_\mathfrak{L},\bar{\lambda}_\mathfrak{L},\eta_\mathfrak{L}^{(\Xi)})$
to the quadruple
$(A_{\mathfrak{L}/\bar{\mathbb{F}}_l},\iota_\mathfrak{L},\bar{\lambda}_\mathfrak{L},\eta_\mathfrak{L}^{(\Xi)}\circ
\hat{\rho}(\alpha))$. The association $\alpha\mapsto
\tilde{\theta}_{\mathfrak{L}}(\alpha)$ gives us an embedding
$M\hookrightarrow \mathrm{End}^0_F(A_{\mathfrak{L}/\bar{F}_l})$.
Since $\mathrm{End}^0_F(A_{\mathfrak{L}/\bar{F}_l})$ is also a CM
quadratic extension of $F$ by Lemma \ref{4.1}, this embedding must
be an isomorphism. Hence we get the desired isomorphism of
$F$-algebras.

\end{proof}

\section{Main result on local indecomposability and applications}
Let $k$ be a number field. Suppose that we are given an abelian variety $A_{/k}$ and an algebra homomorphism $\iota: \mathcal{O}_F\rightarrow\mathrm{End}(A_{/k})$.
Assume that there is a prime ideal $\mathfrak{P}$ of $k$ over a rational prime $p$, such that $A_{/k}$  satisfies
the condition ($\mathrm{NLL}$) in section $1$.
Let $I_\mathfrak{P}$ be the inertia group of $\mathrm{Gal}(\bar{\mathbb{Q}}/k)$ at the prime $\mathfrak{P}$.
Now we can state and prove the main theorem in this paper:

\begin{theorem}\label{Th3}
Under the above notations and assumptions, suppose further that
$A_{/\bar{\mathbb{Q}}}=A_{/k}\times_k \bar{\mathbb{Q}}$ does not
have complex multiplication, then for any $\mathfrak{p}\in
\Sigma_p^{ord}$, the $\mathfrak{p}$-adic Tate module
$T_{\mathfrak{p}}(A)$ of $A$ is indecomposable as an
$I_\mathfrak{P}$-module.
\end{theorem}

\begin{proof}
From Proposition \ref{1.1}, the abelian variety $A_{/\bar{\mathbb{Q}}}$ is isotypic. Hence without loss of generality, we can assume that $A_{/\bar{\mathbb{Q}}}$
is simple.

Fix an embedding $\sigma: F\rightarrow \bar{\mathbb{Q}}$ such that
the composition $i_l\circ \sigma$ induces the prime $\mathfrak{p}$.
From \cite{H12} Proposition $7.1$, the set
$$\{\mathfrak{L}|\mathfrak{L}\text{~is a prime of~} k\text{~over a rational
prime~} l\neq p \text{~such that~} A_{/k} \text{~has good reduction
at~} \mathfrak{L},\text{~and~}\Sigma_{l}^{ord}\neq\emptyset\}$$ has
Dirichlet density $1$. On the other hand, the
 primes $\mathfrak{l}$ in $F$ which splits completely over $\mathbb{Q}$ also has Dirichlet density $1$,
  we can find a prime $\mathfrak{L}$ of $k$ over a rational prime $l$ such that:
\begin{enumerate}
\item $l$ is unramified in $F$;
\item $A_{/k}$ has good reduction at $\mathfrak{L}$ and $\Sigma_l^{ord}$ contains the prime $\mathfrak{l}$ induced by $i_l\circ \sigma$ and
$\mathfrak{l}$ splits over $\mathbb{Q}$.
\end{enumerate}
Let $A_{\mathfrak{L}/\bar{\mathbb{F}}_l}$ be the reduction of $A_{/k}$ at $\mathfrak{L}$, and set $M_{\mathfrak{L}}=\mathrm{End}_F^0(A_{\mathfrak{L}/\bar{\mathbb{F}}_l})$.
By Lemma \ref{4.1}, $M_{\mathfrak{L}}$ is a quadratic CM extension of the field $F$.

Now we use an argument in \cite{H12} Proposition $5.1$ to prove the following statement:
 we can find a prime $\mathfrak{Q}$ of $k$ over a rational prime $q\neq p,l$, such that
\begin{enumerate}
\item $A_{/k}$ has good reduction at $\mathfrak{Q}$;
\item $\Sigma_q^{ord}$ contains the prime induced by $i_q\circ \sigma$;
\item $M_{\mathfrak{Q}}=\mathrm{End}_F^0(A_{\mathfrak{Q}/\bar{\mathbb{F}}_q})$ is a CM quadratic extension of $F$ which is non-isomorphic to $M_{\mathfrak{L}}$.
\end{enumerate}

For simplicity, we use $D$ to denote the division algebra $\mathrm{End}^0(A_{/k})$. Let $Z$ be the center of $D$ (which is a number field by
Proposition \ref{1.2}), and $\mathcal{O}_Z$ be its integer ring. Since we have a homomorphism $F\rightarrow D$, and $F$ is totally real, from
Proposition \ref{1.2}, $Z$ is totally real and either $Z=F=D$ or $D$ is a quaternion division algebra over $Z$ and $[F:Z]=2$.

For any prime $\mathfrak{q}$ of $F$, we fix an isomorphism
$T_{\mathfrak{q}}(A)\cong (\mathcal{O}_{F,\mathfrak{q}})^2$, and
denote by $r_{\mathfrak{q}}: \Gal(\bar{\mathbb{Q}}/k)\rightarrow
GL_2(\mathcal{O}_{F,\mathfrak{q}})$ as the induced Galois
representation on $T_{\mathfrak{q}}(A)$. Let $\mathfrak{q}'$ be the
prime of $Z$ lying under $\mathfrak{q}$, and $Z_\mathfrak{q}$ be the
completion of $Z$ with respect to $\mathfrak{q}'$. Hence
$Z_\mathfrak{q}$ can be identified with the closure of $Z$ in
$F_\mathfrak{q}$. By Tate's conjecture over a number field, which is
proved by Faltings (see \cite{Faltings}), inside the algebra
$\mathrm{End}^0_{Z_\mathfrak{q}}(T_\mathfrak{q}(A))=\mathrm{End}_{\mathcal{O}_{Z,\mathfrak{q}}}(T_\mathfrak{q}(A))\otimes_{\mathcal{O}_{Z,\mathfrak{q}}}Z_\mathfrak{q}$,
the algebra
$C_\mathfrak{q}=Z_\mathfrak{q}[r_\mathfrak{q}(\Gal(\bar{\mathbb{Q}}/k))]$
generated over $Z_\mathfrak{q}$ by the image of $r_\mathfrak{q}$ is
the commutatant
 of $D_\mathfrak{q}=\mathrm{End}^0(A_{/k})\otimes_Z Z_\mathfrak{q}$. Since $D_\mathfrak{q}$ is
 either isomorphic to $Z_\mathfrak{q}$ or a central quaternion algebra over $Z_\mathfrak{q}$,
 $C_\mathfrak{q}$ is also a central simple algebra over $Z_\mathfrak{q}$ and in the Brauer group $Br(Z_\mathfrak{q})$, the class of $D_\mathfrak{q}$ is the inverse
  of the class of $C_\mathfrak{q}$. The order of the class of  $C_\mathfrak{q}$ in $Br(Z_\mathfrak{q})$ is at most $2$, and hence $C_\mathfrak{q}$ is either
isomorphic to  a quaternion division algebra over $Z_\mathfrak{q}$ or  isomorphic to $M_2(Z_\mathfrak{q})$.

We consider the case $\mathfrak{q}=\mathfrak{l}$. Since $A_{/k}$ has goode reduction at $\mathfrak{L}$, we can embed
$M_{\mathfrak{L},\mathfrak{l}}=M_{\mathfrak{L}}\otimes_{F}F_\mathfrak{l}\cong F_\mathfrak{l}\oplus F_\mathfrak{l}$ into $C_{\mathfrak{l}}$.
Hence $C_{\mathfrak{l}}$
is isomorphic to $M_2(F_\mathfrak{l})=M_2(Z_\mathfrak{l})$. Under this assumption,we can apply an argument in \cite{Ri76} Chapter $4$ to prove that the image $\mathrm{Im}(r)$ contains
an open subgroup of $SL_2(\mathbb{Z}_l)\subseteq C_{\mathfrak{l}}^\times$.

Choose a quadratic ramified extension $K/\mathbb{Q}_l$. Since
$F_\mathfrak{l}/\mathbb{Q}_l$ is unramified, $K$ and
$F_\mathfrak{l}$ are linearly disjoint over $\mathbb{Q}_l$. Let $L$
be the compositum field of $K$ and $F_\mathfrak{l}$. Define the
torus $T_{/\mathcal{O}_{F,\mathfrak{l}}}$ of
$GL_{2/\mathcal{O}_{F,\mathfrak{l}}}$ as the norm $1$ subgroup of
$Res_{\mathcal{O}_L/\mathcal{O}_{F,\mathfrak{l}}}(\mathbb{G}_m)$;
i.e.
$$T(\mathcal{O}_{F,\mathfrak{l}})=\{x\in \mathcal{O}_L^\times| Norm_{L/F_\mathfrak{l}}(x)=1\}.$$
Hence $T_{/\mathcal{O}_{F,\mathfrak{l}}}$ is a maximal anisotropic torus of $GL_{2/\mathcal{O}_{F,\mathfrak{l}}}$.
Since
$$T(\mathcal{O}_{F,\mathfrak{l}})\cap SL_2(\mathbb{Z}_l)=\{x\in \mathcal{O}_K^\times| Norm_{K/\mathbb{Q}_l}(x)=1\},$$
$T(\mathcal{O}_{F,\mathfrak{l}})\cap SL_2(\mathbb{Z}_l)$ is a maximal anisotropic torus of $GL_{2/\mathbb{Z}_l}$.

Choose $\alpha\in T(\mathcal{O}_{F,\mathfrak{l}})\cap
\mathrm{Im}(r)\cap SL_2(\mathbb{Z}_l)$, such that $\alpha$ has two
different eigenvalues in $\bar{\mathbb{Q}}_l$. Then
$T(\mathcal{O}_{F,\mathfrak{l}})$ is the centralizer $T_\alpha$ of
$\alpha$ in $GL_2(\mathcal{O}_{F,\mathfrak{l}})$. Since the
isomorphism classes of maximal torus in
$GL_{2/\mathcal{O}_{F,\mathfrak{l}}}$ is finite, the isomorphism
class of the centralizer of $\alpha$ is determined by $\alpha$ mod
$p^j$, for some integer $j$ large enough. In other word, if
$\beta\in SL_2(\mathbb{Z}_l)$, such that $\alpha\equiv\beta$ mod $p^j$,
then the centralizer $T_\beta$ of $\beta$ is isomorphic to
$T_\alpha=T$. By Chebotarev density, we can find a prime
$\mathfrak{Q}$ of $k$ over a rational prime $q\neq p,l$, such that
$A_{/k}$ has good reduction at  $\mathfrak{Q}$ and
$r(Frob_\mathfrak{Q})\equiv \alpha$ mod $p^j$. Hence the commutator
$T_{r(Frob_\mathfrak{Q})}$ of $r(Frob_\mathfrak{Q})$ is isomrphic to
$T$. Let $M_\mathfrak{Q}$ be the field generated over $F$ by the
eigenvalues of $r(Frob_\mathfrak{Q})$. By the above construction,
$\mathfrak{l}$ does not split in $M_\mathfrak{Q}$, and hence
$M_\mathfrak{Q}$ is not isomorphic to $M_\mathfrak{L}$. Further by
\cite{H12} Proposition $7.1$, we can assume that $\Sigma_q^{ord}$
contains the prime induce from $i_q\circ \sigma$.

Now define a finite set of primes $\Xi=\{p,q,l\}$. Hence the abelian
variety $A_{/k}$ can be extended to an abelian scheme
$A_{/\widetilde{\mathcal{W}}_k}$. From Proposition \ref{2.6},
replacing $A_{/k}$ by an isogeny if necessary, we can assume that
the abelian scheme $A_{/\widetilde{\mathcal{W}}_k}$ admits an
$\mathcal{O}_F$-action $\iota: \mathcal{O}_F\rightarrow
\mathrm{End}(A_{/\widetilde{\mathcal{W}}_k})$ and a
$\mathfrak{c}$-polarization $\phi$ for some fractional ideal
$\mathfrak{c}$ of $F$.
 Then by choosing a integral level structure
$\alpha^{\Xi}$ of $A$, we get a quadruple
$(A_{/\widetilde{\mathcal{W}}_k},\iota,\phi,\alpha^{\Xi})$, which
represents a point in the Shimura variety $x\in
Sh^{(\Xi)}(\widetilde{\mathcal{W}}_k)$.

Now assume that the Tate module $T_{\mathfrak{p}}(A)$ is decomposable as an $I_\mathfrak{P}$-module. Then the exact sequence of Barsotti-Tate  $\mathcal{O}_{\mathfrak{p}}$-modules over $\widetilde{W}_p$:
$$0\rightarrow\mu_{p^\infty}\otimes_{\mathbb{Z}_p}\mathcal{O}_{\mathfrak{p}} \rightarrow A[\mathfrak{p}^\infty]\rightarrow F_{\mathfrak{p}}/\mathcal{O}_{\mathfrak{p}}\rightarrow 0$$
splits. Then by Theorem \ref{Th2}, we must have isomorphisms of $F$-algebras: $M_{\mathfrak{Q}}\cong\mathrm{End}_F^0(A_{\mathfrak{Q}/\bar{\mathbb{F}}_p})$
and $M_{\mathfrak{L}}\cong\mathrm{End}_F^0(A_{\mathfrak{L}/\bar{\mathbb{F}}_p})$. But this contradicts with our construction $M_{\mathfrak{Q}}\ncong M_{\mathfrak{L}}$.
Hence $T_{\mathfrak{p}}(A)$ must be indecomposable as an $I_\mathfrak{P}$-module.
\end{proof}
\subsection{Application to Hilbert modular Galois represenations}
As the first application of Theorem \ref{Th3}, we study the Galois
representation attached to certain Hilbert modular forms. First we
recall the notions of Hilbert modular forms and Hecke operators.

Let $\mathrm{I}=\mathrm{Hom}_\mathbb{Q}(F,\bar{\mathbb{Q}})$, and
let $\mathbb{Z}[\mathrm{I}]$ be the set of formal
$\mathbb{Z}$-linear combinations of elements in $\mathrm{I}$. Then
$\mathbb{Z}[\mathrm{I}]$ can be identified with the character group
$X(T)$ of the torus $T$. Take $k=(k_\sigma)_{\sigma\in \mathrm{I}}$
such that $k_\sigma\geq 2$ for all $\sigma\in \mathrm{I}$ and all
the $k_\sigma$'s have the same parity. Set  $t=(1,\ldots,1)\in
\mathbb{Z}[\mathrm{I}]$ and $n=k-2t$. Choose
$v=(v_\sigma)_{\sigma\in \mathrm{I}}$ such that $v_\sigma\geq 0$,
for all $\sigma$, $v_\sigma=0$ for at least one $\sigma$, and there
exists $\mu\in \mathbb{Z}$ such that $n+2v=\mu t\in
\mathbb{Z}[\mathrm{I}]$. Then define $w=v+k-t$.

Recall that in Section $3$ we define the algebraic group
$G=Res_{\mathcal{O}_F/\mathbb{Z}}(GL_2)$ and
$T=Res_{\mathcal{O}_F/\mathbb{Z}}(\mathbb{G}_m)$. Denote by
$\nu:G\rightarrow T$ the reduced norm morphism. Fix an open subgroup
$U$ of $G(\widehat{\mathbb{Z}})=GL_2(\widehat{\mathcal{O}}_F)$ where
$\widehat{\mathcal{O}}_F=\mathcal{O}_F\otimes_\mathbb{Z}\widehat{\mathbb{Z}}=\Pi_{\mathfrak{p}}\mathcal{O}_{F,\mathfrak{p}}$.
In the last product, $\mathfrak{p}$ ranges over all the prime ideals
of $\mathcal{O}_F$ and $\mathcal{O}_{F,\mathfrak{p}}$ is the
completion of $\mathcal{O}_F$ at $\mathfrak{p}$. Let
$F_\mathbb{A}=F\otimes_\mathbb{Z}\mathbb{A}$ be the adele ring of
$F$. We can decompose the group $G(F_\mathbb{A})$ as the product
$G_\infty\times G_f$, where $G_\infty$ (resp. $G_f$) is the infinite
(resp. finite) part of $G(F_\mathbb{A})$, and for each $u\in
G(F_\mathbb{A})$, we have the corresponding decomposition
$u=u_\infty u_f$.

Let $\mathfrak{h}$ be the complex upper half plane and
$i=\sqrt{-1}\in \mathfrak{h}$. Let $\mathfrak{h}^\mathrm{I}$ be the
product of $d$ copies of $\mathfrak{h}$ indexed by elements in
$\mathrm{I}$ and $z_0=(i,\ldots,i)\in \mathfrak{h}^\mathrm{I}$.
Define a function $j:G_\infty\times
\mathfrak{h}^\mathrm{I}\rightarrow\mathbb{C}^\mathrm{I}$ by the
formula:
$$\left(\mat{a_\tau}{b_\tau}{c_\tau}{d_\tau},z_\tau\right)_{\tau\in \mathrm{I}}\mapsto (c_\tau z_\tau+d_\tau)_{\tau\in \mathrm{I}}.$$
\begin{definition}
Define the space of Hilbert modular cusp forms
$S_{k,w}(U;\mathbb{C})$ as the set of functions
$f:G(F_\mathbb{A})\rightarrow \mathbb{C}$ satisfying the following
conditions:
\begin{enumerate}
\item $f|_{k,w}u=f$, for all $u\in UC_{\infty+}$ where $C_{\infty+}=(\mathbb{R}^\times\cdot SO_2(\mathbb{R}))^\mathrm{I}\subseteq
G_\infty$, and
$$f|_{k,w}u(x)=j(u_\infty,z_0)^{-k}v(u_\infty)^w f(xu^{-1});$$
\item $f(ax)=f(x)$ for all $a\in G(\mathbb{Q})=GL_2(F)$;
\item For any $x\in G_f$, the function
$f_x:\mathfrak{h}^\mathrm{I}\rightarrow\mathbb{C}$ defined by
$u_\infty(z_0)\mapsto j(u_\infty,z_0)^kv(u_\infty)^{-w}f(xu_\infty)$
for $u_\infty\in G_\infty$ is holomorphic;
\item $\int_{F_\mathbb{A}/F}f\left(\mat{1}{a}{0}{1}x\right)da=0$ for all $x\in
G(F_\mathbb{A})$ and additive Haar measure $da$ on $F_\mathbb{A}/F$.
\end{enumerate}
When $F=\mathbb{Q}$, we also add the following condition: the
function $|\mathrm{Im}(z)^{k/2}f_x(z)|$ is uniformly bounded on
$\mathfrak{h}$ for all $x\in G_f=GL_2(\mathbb{A}_f)$.
\end{definition}

Fix an integral ideal $\mathfrak{m}$ of $F$, we define three open
subgroups of $GL_2(\widehat{\mathcal{O}}_F)$:
$$U_0(\mathfrak{m})=\left\{\mat{a}{b}{c}{d}\in GL_2(\widehat{\mathcal{O}}_F)|c\in\mathfrak{m}\widehat{\mathcal{O}}_F\right\},$$
$$U_1(\mathfrak{m})=\left\{\mat{a}{b}{c}{d}\in GL_2(\widehat{\mathcal{O}}_F)|c\in\mathfrak{m}\widehat{\mathcal{O}}_F,
a\equiv 1 ~mod~ \mathfrak{m}\widehat{\mathcal{O}}_F\right\},$$
$$U(\mathfrak{m})=\left\{\mat{a}{b}{c}{d}\in GL_2(\widehat{\mathcal{O}}_F)|c\in\mathfrak{m}\widehat{\mathcal{O}}_F,
a\equiv d\equiv 1 ~mod~
\mathfrak{m}\widehat{\mathcal{O}}_F\right\},$$ and set
$S_{k,w}(\mathfrak{m},\mathbb{C})=S_{k,w}(U_1(\mathfrak{m}),\mathbb{C})$.

Let $U, U'$ be two open compact subgroups of $G_f$ and fix $x\in
G_f$. Define a Hecke operator
$$[UxU']: S_{k,w}(U;\mathbb{C})\rightarrow S_{k,w}(U';\mathbb{C}), f\mapsto \sum_i f|_{k,w}x_i,$$
where $\{x_i\}$ is a set of representatives of the left cosets
$U\backslash UxU'$; i.e., we have  $UxU'=\coprod U x_i$ and when we
consider the action $f|_{k,w}x_i$, we regard $x_i\in G_f$ as an
element in $G(F_\mathbb{A})$ such that its infinite part consists of
$d$ copies of identity matrices. For all prime ideal $\mathfrak{q}$
of $F$, fix a uniformizer $\pi_\mathfrak{q}$ of $F_\mathfrak{q}$,
and define the Hecke operator
$$T(\mathfrak{q})=\left[U\mat{1}{0}{0}{\beta_\mathfrak{q}}U\right]:S_{k,w}(U;\mathbb{C})\rightarrow
S_{k,w}(U;\mathbb{C}),$$ where $\beta_\mathfrak{q}\in
F_{\mathbb{A}_f}^\times$ is the finite idele whose
$\mathfrak{q}$-component is $\pi_\mathfrak{q}$ and all the other
components are $1$. For each fractional ideal $\mathfrak{n}$ of $F$,
set
$\alpha=\Pi_\mathfrak{q}\pi_\mathfrak{q}^{v_\mathfrak{q}(\mathfrak{n})}\in
F_{\mathbb{A}_f}^\times$, and define the Hecke operator
$$\langle\mathfrak{n}\rangle=\left[U\mat{\alpha}{0}{0}{\alpha}U\right]:S_{k,w}(U;\mathbb{C})\rightarrow
S_{k,w}(U;\mathbb{C}).$$

Let $f\in S_{k,w}(\mathfrak{m},\mathbb{C})$ be a normalized Hilbert
modular eigenform in the sense that  for any prime ideal
$\mathfrak{q}$ of $F$,  there exists $c(\mathfrak{q},f)\in
\bar{\mathbb{Q}}$ and $d(\mathfrak{q},f)\in \bar{\mathbb{Q}}$ such
that $T(\mathfrak{q})(f)=c(\mathfrak{q},f)\cdot f$ and
$\langle\mathfrak{q}\rangle(f)=d(\mathfrak{q},f)\cdot f$. Let $K_f$
be the field generated over $\mathbb{Q}$ by all the
$c(\mathfrak{a},f)$'s and $d(\mathfrak{a},f)$'s. Shimura proved that
$K_f$ is a number field which is either totally real or CM. Denote
by $\mathcal{O}_f$ the integer ring of $K_f$.

For such an $f$, let $\pi_f=\otimes \pi_v$ be the automorphic
representation of $GL_2(F_{\mathbb{A}_f})$ on the linear span of all
the right translations of $f$ by elements of
$GL_2(F_{\mathbb{A}_f})$, here $F_{\mathbb{A}_f}$ is the finite
adele of $F$,and $\pi_v$ is a representation of $GL_2(F_v)$ for each
finite place $v$ of $F$. We assume that one of the following two
statements holds:
\begin{enumerate}
\item $[F:\mathbb{Q}]$ is odd;
\item there exists some finite place $v$ of $F$ such that $\pi_v$ is square integrable.
\end{enumerate}
For such an eigenform $f$, the following result is known (see
\cite{H06} Theorem $2.43$ for details and historical remarks). For
each prime $\lambda$ of $\mathcal{O}_f$ over a rational prime $p$,
there is a continuous representation $\rho_{f,\lambda}:
\mathrm{Gal}(\bar{\mathbb{Q}}/F)\rightarrow
GL_2(\mathcal{O}_{f,\lambda})$, which is unramified outside primes
dividing $\mathfrak{m}p$ such that for any primes $\mathfrak{q}\nmid
\mathfrak{m}p$, we have:
$$\textrm{trace} (\rho_{f,\lambda}(Frob_\mathfrak{q}))=c(\mathfrak{q},f), \text{~and~}
\det(\rho_{f,\lambda}(Frob_\mathfrak{q}))=d(\mathfrak{q},f)N\mathfrak{q}.$$
Here $\mathcal{O}_{f,\lambda}$ is the completion of $\mathcal{O}_f$
at $\lambda$, $Frob_\mathfrak{q}$ is the Frobenius of
$\mathrm{Gal}(\bar{\mathbb{Q}}/F)$ at $\mathfrak{q}$,and for any
ideal $\mathfrak{b}$ of $\mathcal{O}_F$, $N\mathfrak{b}$ is the
cardinality number of the ring $\mathcal{O}_F/\mathfrak{b}$.

Fix a prime $\mathfrak{p}$ of $\mathcal{O}_F$ over a rational prime $p$, let $D_\mathfrak{p}$(resp. $I_\mathfrak{p}$) be the
decomposition group (resp. inertia group) of $\mathrm{Gal}(\bar{\mathbb{Q}}/F)$ at $\mathfrak{p}$.
Let $\lambda$ be a prime  of $\mathcal{O}_f$ over  $p$.
From \cite{Wi88} Lemma $2.1.5$, if $c(\mathfrak{p},f)$ is a unit mod $\lambda$,
then the restriction of $\rho_{f,\lambda}$ to $D_\mathfrak{p}$ is upper triangular, i.e. there exist two characters $\epsilon_1,\epsilon_2$ of $D_\mathfrak{p}$,
such that
$$\rho_{f,\lambda}|_{D_\mathfrak{p}}\sim \mat{\epsilon_1}{\ast}{0}{\epsilon_2}.$$

We would like to prove the following:

\begin{theorem}\label{Th4}
Under the above notations, suppose that $k=2t$ and $f$ is nearly
$\mathfrak{p}$-ordinary in the sense that $c(\mathfrak{p},f)$ is a
unit mod $\lambda$. If $f$ does not have complex multiplication,
then the representation $\rho_{f,\lambda}|_{I_\mathfrak{p}}$ is
indecomposable.
\end{theorem}

\begin{proof}
As the Hecke operator $T(\mathfrak{p})$ acts nontrivially on $f$,
from \cite{H89} Corollary $2.2$, the local representation
$\pi_\mathfrak{p}$ of $GL_2(F_\mathfrak{p})$ is either a principal
representation $\pi(\xi_\mathfrak{p},\eta_\mathfrak{p})$ or a
special representation $\sigma(\xi_\mathfrak{p},\eta_\mathfrak{p})$.
From the argument in \cite{H89} Section $2$, we can find a finite
character $\chi:
F_\mathbb{A}^\times/F^\times\rightarrow\bar{\mathbb{Q}}^\times$
($F_\mathbb{A}$ is the adele ring of $F$) such that the
$\mathfrak{p}$-component of $\chi$ satisfies
$\chi_\mathfrak{p}=\xi_\mathfrak{p}$ on
$\mathcal{O}_{F,\mathfrak{p}}^\times$ and unramified at every
infinite place of $F$. Then the argument in \cite{H89} Section $2$
implies that the automorphic representation $\chi\otimes\pi$
corresponds to a primitive $\mathfrak{p}$-ordinary newform $f_0$. If
we regard  the representations $\rho_{f,\lambda}$ and
$\rho_{f_0,\lambda}$ as representations in
$GL_2(\bar{\mathbb{Q}}_p)$, then they are related by the formula
$\rho_{f,\lambda}\otimes\chi^{-1}=\rho_{f_0,\lambda}$. It is enough
to prove the statement for the newform $f_0$ and henceforth we
assume that the Hilbert modular form $f$ is a primitive
$\mathfrak{p}$-ordinary newform with character $\psi$ for some idele
class character $\psi$ of $F$ with finite order.

From \cite{H81} Theorem $4.4$ or \cite{Wi86} Theorem $2.1$, there
exists an abelian variety $A_f$ defined over $F$ , a finite
extension $L/K_f$ whose degree equals to the dimension of $A_f$ and
an embedding $\theta: L\rightarrow \mathrm{End}(A_{f/F})$ such that
the $\lambda$-adic representation associated to the Tate module of
$A_f$ is isomorphic to $\rho_{f,\lambda}$. Moreover the number field
$L$ is either totally real or CM. To be more precise, there exists
an integer $e$ such that $dim(A_{f/F})=e[K_f:\mathbb{Q}]$. When
$[F:\mathbb{Q}]$ is odd, $e=1$ and there is nothing to explain in
this situation. When $[F:\mathbb{Q}]$ is even, $e$ can be bigger
than $1$, and a priori the $p$-adic Tate module of $A_{f/F}$ gives
us a representation of $\Gal(\bar{\mathbb{Q}}/F)$ in
$GL_2(L_\lambda)$,where $L_\lambda$ is a finite extension of
$K_{f,\lambda}$. Since this representation is odd, by choosing
suitable eigenvectors of a complex conjugation $c\in
\Gal(\bar{\mathbb{Q}}/F)$ as basis for $T_p(A_f)$, we can realize
this representation in $GL_2(K_{f,\lambda})$. (See \cite{Wi88}
Section $2.1$ for details.)

As $c(\mathfrak{p},\lambda)$ is a unit mod $\lambda$, the abelian
variety $A_f$ has potentially semistable reduction at $\mathfrak{p}$
by the lemma in \cite{Wi86} Section $2$. More precisely, if we
denote by $F_\psi$ the number field corresponding to the character
$\psi$ by class field theory,
 then $A_f$ has semistable reduction  over $F_\psi$.
In fact,  choose a prime $\lambda'$ of $\mathcal{O}_f$ over a
rational prime $l\neq p$ and consider the $\lambda'$-adic
representation $\rho_{f,\lambda'}$. When $\mathfrak{p}$ does not
divide the level $\mathfrak{m}$, the abelian variety $A_f$ has good
reduction at $\mathfrak{p}$ because the representation
$\rho_{f,\lambda'}$ is unramified at $\mathfrak{p}$. If
$\mathfrak{p}$ divides $\mathfrak{m}$, one can consider the  complex
representation $\sigma_\mathfrak{p}$ of the local Weil-Deligne group
$W'_{F_\mathfrak{p}}$ of $F$ at $\mathfrak{p}$ associated to
$\rho_{f,\lambda'}$ (see \cite{T}). Then by a result of Carayol
\cite{Ca}, we have an isomorphism $\pi(\sigma_\mathfrak{p})\cong
\pi_\mathfrak{p}$, where $\pi(\sigma_\mathfrak{p})$ is the
representation of $GL_2(F_\mathfrak{p})$ associated to
$\sigma_\mathfrak{p}$ under the local Langlands correspondence. In
particular, the Euler factor $L(\pi_\mathfrak{p},s)$ of the
$L$-series at $\mathfrak{p}$ is given by
$(1-c(\mathfrak{p},f)N\mathfrak{p}^{-s})^{-1}$. As
$c(\mathfrak{p},f)\neq 0$ by assumption, $L(\pi_\mathfrak{p},s)$ is
nontrivial. Hence $\pi_\mathfrak{p}$ is either a special
representation $\sigma(\alpha_\mathfrak{p},\beta_\mathfrak{p})$ or a
principal series representation
$\pi(\alpha_\mathfrak{p},\beta_\mathfrak{p})$, where
$\alpha_\mathfrak{p},\beta_\mathfrak{p}$ are two quasi-characters of
$F_\mathfrak{p}^\times$. In the first case, from \cite{Wi86} Theorem
$2.2$,
 the reduction of $A_f$ at $\mathfrak{p}$ is purely multiplicative. From the uniformization result in \cite{Mum72}, $\rho_{f,\lambda}|_{I_\mathfrak{p}\cap \Gal(\bar{\mathbb{Q}}/F_\psi)}$
is indecomposable. As $I_\mathfrak{p}\cap
\Gal(\bar{\mathbb{Q}}/F_\psi)$ is a subgroup of $I_\mathfrak{p}$
with finite index, and $char(K_f)=0$, the representation
$\rho_{f,\lambda}|_{I_\mathfrak{p}}$ is also indecomposable. In the
second case, as the Euler factor $L(\pi_\mathfrak{p},s)\neq 1$, one
of the quasi-characters $\alpha_\mathfrak{p},\beta_\mathfrak{p}$ is
unramified. By comparing the determinant of the two representations
$\pi_\mathfrak{p}$ and $\sigma_\mathfrak{p}$, we see that the
product
$\psi_\mathfrak{p}^{-1}\alpha_\mathfrak{p}\beta_\mathfrak{p}$ is
unramified, where $\psi_\mathfrak{p}$ is the
$\mathfrak{p}$-component of the idele class character $\psi$. Hence
over $F_\psi$, both quasi-characters $\alpha_\mathfrak{p}$ and
$\beta_\mathfrak{p}$ are unramified. Then from the criterion of
N\'{e}ron-Ogg-Shafarevich, the abelian variety $A_f$ has good
reduction over $F_\psi$ at $\mathfrak{p}$.

From now on we assume that $A_f$ has good reduction over $F_\psi$. From Proposition \ref{1.2}, we see that $A_{f/\bar{\mathbb{Q}}}$ is isotypic; i.e. there
exists a simple abelian variety $B_{/\bar{\mathbb{Q}}}$ such that there exists an isogeny $\varphi: A_f\rightarrow B^e$ for some integer $e\geq 1$. This
isogeny induces an isomorphism of simple algebras $i: \mathrm{End}^0(A_{f/\bar{\mathbb{Q}}})\rightarrow \mathrm{End}^0((B_{/\bar{\mathbb{Q}}})^e)$. Hence we
have an embedding $\theta_B=i\circ \theta: L\rightarrow \mathrm{End}^0((B_{/\bar{\mathbb{Q}}})^e)$.

From
Proposition \ref{1.4},  we can find a totally real field $F_B$ and a homomorphism $\iota_B: F_B\rightarrow \mathrm{End}^0(B_{/\bar{\mathbb{Q}}})$, such
that $[F_B:\mathbb{Q}]=dim(B_{/\bar{\mathbb{Q}}})$. Let $Z$ be the center of the division algebra $\mathrm{End}^0(B_{/\bar{\mathbb{Q}}})$. From the proof
of Proposition \ref{1.4}, if we identify $F_B$ as a subalgebra of $\mathrm{End}^0(B_{/\bar{\mathbb{Q}}})$ by $\iota_B$, then $Z\subseteq F_B$ and
$[F_B: Z]\leq 2$.

If $[F_B: Z]=1$, we have $F_B=Z$ and hence  $F_B\subseteq \theta_B(L)$. Since both $A_f$ and $B$ are projective varieties, we can find a finite extension
$M$ of $F_\psi$ such that
\begin{enumerate}
\item the abelian variety $B$ is defined over $M$;
\item we have the equalities of endomorphism algebras: $\mathrm{End}(A_{f/\bar{\mathbb{Q}}})=\mathrm{End}(A_{f/M})$ and $\mathrm{End}(B_{/\bar{\mathbb{Q}}})
=\mathrm{End}(B_{/M})$.
\item the isogeny $\varphi$ is defined over $M$.
\end{enumerate}
Under the above notations, the isogeny $\varphi$ gives an
isomorphism of $p$-adic Tate modules $T_p(B)\otimes_{F_B}L\cong
T_p(A)$, which is equivariant under the action of the Galois group
$\Gal(\bar{\mathbb{Q}}/M)$.

If $[F_B: Z]=2$, $F_B$ may not contained in the image $\theta_B(L)$. In this case, we can find a quadratic extension $K/L$ such that $F_B$ can be embedded
 into $K$. As the homomorphism $\theta: L\rightarrow \mathrm{End}^0(A_{f/\bar{\mathbb{Q}}})$ identifies $L$ with a maximal commutative subfield of the
 simple algebra $\mathrm{End}^0(A_{f/\bar{\mathbb{Q}}})$, we can extend this homomorphism to a homomorphism $\theta': K\rightarrow \mathrm{End}^0(A^2_{f/\bar{\mathbb{Q}}})$,
 which identifies $K$ with a maximal commutative subfield of $\mathrm{End}^0(A^2_{f/\bar{\mathbb{Q}}})$. Similarly we can extend the homomorphism $\theta_B$
 to a homomorphism $\theta'_B:K\rightarrow \mathrm{End}^0(B^{2e}_{/\bar{\mathbb{Q}}})$. Since $A^2_{f/\bar{\mathbb{Q}}}$ is isogeneous to $B^{2e}_{/\bar{\mathbb{Q}}}$,
 the simple algebras $\mathrm{End}^0(A^2_{f/\bar{\mathbb{Q}}})$ and $\mathrm{End}^0(B^{2e}_{/\bar{\mathbb{Q}}})$ are isomorphic. Since all automorphisms
 of a simple algebra are inner, by choosing a suitable isogeny from $A^2_{f/\bar{\mathbb{Q}}}$ to $B^{2e}_{/\bar{\mathbb{Q}}}$, we have an isomorphism
 $i':\mathrm{End}^0(A^2_{f/\bar{\mathbb{Q}}})\cong \mathrm{End}^0(B^{2e}_{/\bar{\mathbb{Q}}})$, such that
 $i'\circ \theta'=\theta'_B:K\cong
 \mathrm{End}^0(B^{2e}_{/\bar{\mathbb{Q}}})$.

 By the same argument as above, we can find a finite extension $M/F_\psi$ such that we have an isomorphism of $p$-adic Tate modules:
$T_p(B)\otimes_{F_B}K\cong T_p(A_f)\otimes_L K$, which is
equivariant under the action of $\Gal(\bar{\mathbb{Q}}/M)$.

As $B^e_{/M}$ is isogenous to $A_{f/M}$, $B_{/M}$ has good reduction
at a prime $\mathfrak{p}'$ of $M$ over the prime $\mathfrak{p}$ of
$F$. By Theorem \ref{Th3}, for any place $\lambda_B$ of $F_B$ such
that the $\lambda_B$-divisible Barsotti-Tate module of $B_{/M}$ is
ordinary, the corresponding $\lambda_B$-adic Tate module is
indecomposable as a $\Gal(\bar{\mathbb{Q}}/M)\cap
I_\mathfrak{p}$-module.  By the above isomorphism of Tate modules,
we see that $\rho_{f,\lambda}|_{Gal(\bar{\mathbb{Q}}/M)\cap
I_\mathfrak{p}}$ is indecomposable. Since
$\Gal(\bar{\mathbb{Q}}/M)\cap I_\mathfrak{p}$ is a subgroup of
$I_\mathfrak{p}$ with finite index, and $char(K_f)=0$, the
representation $\rho_{f,\lambda}|_{I_\mathfrak{p}}$ must be also
 indecomposable.
\end{proof}
From Theorem \ref{Th4}, we can prove a result on local
indecomposability of $\Lambda$-adic Galois representations. First we
briefly recall the definition of ordinary Hecke algebras defined in
\cite{H88} Section $3$.

Let $\Phi$ be the Galois closure of $F$ in $\bar{\mathbb{Q}}$. The
embedding $i_p:\bar{\mathbb{Q}}\rightarrow \bar{\mathbb{Q}}_p$
induces a $p$-adic valuation on $\Phi$ and we denote by
$\mathcal{O}_\Phi$ the valuation ring. Let $K$ be a finite extension
of the $p$-adic closure of $\Phi$ in $\bar{\mathbb{Q}}_p$, and
$\mathcal{O}_K$ be the valuation ring of $K$. Let $F_\infty/F$ be
the maximal abelian extension of $F$ unramified outside $p$ and
$\infty$, and $Z$ be its Galois group. Let $Z_1$ be the torsion free
part of $Z$. Let $\Lambda=\mathcal{O}_K[[ Z_1]]$ be the continuous
group algebra of $Z_1$ over $\mathcal{O}_K$. Then $\Lambda$ is
(noncanonically) isomorphic to the formal power series ring of
$1+\delta$ variables over $\mathcal{O}_K$, where $\delta$ is the
defect in Leopoldt's conjecture. Let $\chi:
\Gal(\bar{\mathbb{Q}}/F)\rightarrow \mathbb{Z}_p^\times$ be the
cyclotomic character. The restriction of $\chi$ to $Z_1$ gives a
character of $Z_1$, which is still denoted by $\chi$. For any
integer $k\geq 2$ and a finite order character
$\epsilon:Z_1\rightarrow \bar{\mathbb{Q}}_p$. The character
$\epsilon\chi^{k-1}:Z_1\rightarrow \bar{\mathbb{Q}}_p$ gives a
homomorphism $\kappa_{k,\epsilon}:\Lambda\rightarrow
\bar{\mathbb{Q}}_p$.

For any two open compact subgroups $U,U'$ of $G_f$ and $x\in G_f$,
we have the modified Hecke operator defined in \cite{H88} Section
$3$:
$$(UxU'): S_{k,w}(U;\mathbb{C})\rightarrow S_{k,w}(U';\mathbb{C}).$$
For each prime ideal $\mathfrak{q}$ of $F$, set
$$T_0(\mathfrak{q})=\left(U\mat{1}{0}{0}{\beta_\mathfrak{q}}U\right):S_{k,w}(U;\mathbb{C})\rightarrow
S_{k,w}(U;\mathbb{C}),$$ where $\beta_\mathfrak{q}$ is the same as
in the definition of $T(\mathfrak{q})$.

Fix an integral ideal $\mathfrak{n}$ of $F$ which is prime to $p$,
and for each integer $\alpha\geq 1$, set
$S_{k,w}(\mathfrak{n}p^\alpha;\mathbb{C})=S_{k,w}(U_1(\mathfrak{n}\cap
U(p^\alpha));\mathbb{C})$. Define the Hecke algebra
$h_{k,w}(\mathfrak{n}p^\alpha;\mathcal{O}_\Phi)$ as the
$\mathcal{O}_\Phi$-subalgebra of
$\mathrm{End}_\mathbb{C}(S_{k,w}(\mathfrak{n}p^\alpha;\mathbb{C}))$
generated by all the $T_0(\mathfrak{q})$'s over $\mathcal{O}_\Phi$
and define
$h_{k,w}(\mathfrak{n}p^\alpha;\mathcal{O}_K)=h_{k,w}(\mathfrak{n}p^\alpha;\mathcal{O}_\Phi)\otimes_{\mathcal{O}_\Phi}\mathcal{O}_K$.
Inside $h_{k,w}(\mathfrak{n}p^\alpha;\mathcal{O}_K)$ we have the
$p$-adic ordinary projector
$e_\alpha=\lim_{n\rightarrow\infty}T_0(p)^{n!}$ and we have the
ordinary Hecke algebra
$h_{k,w}^{ord}(\mathfrak{n}p^\alpha;\mathcal{O}_K)=e_\alpha
h_{k,w}(\mathfrak{n}p^\alpha;\mathcal{O}_K)$. For
$\beta\geq\alpha\geq 0$, we have a natural surjective
$\mathcal{O}_K$-algebra homomorphism
$h_{k,w}^{ord}(\mathfrak{n}p^\beta;\mathcal{O}_K)\rightarrow
h_{k,w}^{ord}(\mathfrak{n}p^\alpha;\mathcal{O}_K)$, and we define
$$h_{k,w}^{ord}(\mathfrak{n}p^\infty;\mathcal{O}_K)=\lim_{\longleftarrow_\alpha}h_{k,w}^{ord}(\mathfrak{n}p^\alpha;\mathcal{O}_K).$$
Form \cite{H88} Theorem $3.3$, the ordinary Hecke algebra
$h_{k,w}^{ord}(\mathfrak{n}p^\infty;\mathcal{O}_K)$ is a torsion
free $\Lambda$-module of finite type, and the isomorphism class of
$h_{k,w}^{ord}(\mathfrak{n}p^\infty;\mathcal{O}_K)$ as an
$\mathcal{O}_K$-algebra only depends on the class of $v$ in
$\mathbb{Z}[\mathrm{I}]/\mathbb{Z}t$, and hence we denote this
algebra by $h_v^{ord}(\mathfrak{n}p^\infty;\mathcal{O}_K)$.

Now set $\mathbf{h}=h^{ord}_0(\mathfrak{n}p^\infty;\mathcal{O}_K)$.
Fix $\mathrm{Spec}(\Lambda_L)\rightarrow \mathrm{Spec}(\mathbf{h})$
a (reduced) irreducible component of $\mathbf{h}$ and let
$\mathcal{F}:\mathbf{h}\rightarrow \Lambda_L$ be the corresponding
homomorphism. Then $\Lambda_L$ is finite
 free over $\Lambda$, and the quotient field $L$ of $\Lambda_L$ is a finite extension
of the quotient field of $\Lambda$. Let $P$ be a
$\bar{\mathbb{Q}}_p$-valued point of $\Lambda_L$, and let
$\varphi_P:\Lambda_L\rightarrow \bar{\mathbb{Q}}_p$ be the
corresponding homomorphism. The point $P$ is called an arithmetic
point if $\varphi_P$ is an extension of $\kappa_{k,\epsilon}$ for
some $k$ and $\epsilon$. If $P$ is an arithmetic point, then the
composition $\varphi_P\circ \mathcal{F}: \mathbf{h}\rightarrow
\bar{\mathbb{Q}}_p$ gives the Hecke eigenvalues of a classical
Hilbert modular form $f$ of weight $k$ and tame level
$\mathfrak{n}$. We also say that the Hilbert modular form $f$
corresponds to $P$, and $f$ belongs to the family $\mathcal{F}$. We
say that $\mathcal{F}$ has complex multiplication if there exists an
arithmetic point $P$ in $\mathcal{F}$, such that the corresponding
Hilbert modular form has complex multiplication. Once this is the
case, then for all arithmetic point  in $\mathcal{F}$, the
corresponding Hilbert modular form also has complex multiplication.

It's well known that there is a $2$-dimensional Galois
representation $\rho_{\mathcal{F}}:
\Gal(\bar{\mathbb{Q}}/F)\rightarrow GL_2(L)$ attached to
$\mathcal{F}$ such that for each prime $\mathfrak{p}$ of $F$ over
$p$, the restriction of $\rho_{\mathcal{F}}$ to the decomposition
$D_{\mathfrak{p}}$ is upper triangula;i.e.
$\rho_{\mathcal{F}}|_{D_{\mathfrak{p}}}$ is of the shape:
\begin{equation*}
\rho_{\mathcal{F}}|_{D_{\mathfrak{p}}}\sim
\mat{\delta_\mathfrak{p}}{u_\mathfrak{p}}{0}{\varepsilon_\mathfrak{p}},
\end{equation*}
here $\delta_\mathfrak{p},\varepsilon_\mathfrak{p}:D_{\mathfrak{p}}\rightarrow \Lambda_L$ are two characters of $D_{\mathfrak{p}}$.

\begin{theorem}\label{Th5}
Suppose that  $\mathcal{F}$ does not have complex multiplication,
and $\mathcal{F}$ has an arithmetic point $P$ which corresponds to a
weight $2$ Hilbert modular form satisfying the condition required in
Theorem \ref{Th4}. Then there exists a proper closed subscheme $S$
of $\mathrm{Spec}(\Lambda_L)$ such that for all arithmetic points
$P$ of $\mathrm{Spec}(\Lambda_L)$ outside $S$ which corresponds to a
classical form $f$, the representation $\rho_f|_{D_{\mathfrak{p}}}$
is indecomposable, where $\rho_f$ is the Galois representation
attached to $f$. In particular, when Leopoldt conjecture holds for
$F$ and $p$, then for all but finitely many classical forms $f$
belonging to $\mathcal{F}$, the representation
$\rho_f|_{D_{\mathfrak{p}}}$ is indecomposable.
\end{theorem}
The proof follows essentially from the argument in \cite{GV} Theorem
$18$. For the sake of completeness, we give a proof here.
\begin{proof}
By the assumption and Theorem \ref{Th4}, the representation
$\varphi_P\circ\rho_{\mathcal{F}}|_{D_{\mathfrak{p}}}$ is
indecomposable. Hence $\rho_{\mathcal{F}}|_{D_{\mathfrak{p}}}$ is
indecomposable either. Define
$c_\mathfrak{p}=\varepsilon^{-1}_\mathfrak{p}\cdot u_\mathfrak{p}:
D_{\mathfrak{p}}\rightarrow \Lambda_L$. Then it's easy to check that
$c_\mathcal{F}$ satisfies the  cocycle condition and
$\rho_{\mathcal{F}}|_{D_{\mathfrak{p}}}$ is indecomposable if and
only if the class $[c_\mathfrak{p}]$ of $c_\mathfrak{p}$ in
$H^1(D_{\mathfrak{p}},
\Lambda_L(\delta_\mathfrak{p}\varepsilon^{-1}_\mathfrak{p}))$ is
nontrivial. Since $\Lambda_L$ is finite over $\Lambda$, the residue
field of $\Lambda_L$ is finite and let $q$ be its order. Let $E_1$
be the compositum of the finitely many tamely ramified abelian
extension of $F_\mathfrak{p}$ whose order divides $q-1$, and $E_2$
be the maximal abelian pro-$p$-extension of $F_\mathfrak{p}$. Denote
by $E$ the compositum field of $E_1$ and $E_2$ and set
$H=\Gal(\bar{\mathbb{Q}}_p/E)\subseteq D_\mathfrak{p}$. Then the
characters $\delta_\mathfrak{p}$ and $\varepsilon_\mathfrak{p}$ are
trivial when restricted to $H$. Hence the restriction of
$\rho_\mathcal{F}$ to $H$ is of the shape:

$$
\rho_{\mathcal{F}}|_{H}\sim \mat{1}{\lambda}{0}{1},
$$
for some (additive) homomorphism $\lambda:H\rightarrow\Lambda_L$.
From \cite{GV} Lemma $19$, the restriction
$$H^1(D_\mathfrak{p},\Lambda_L(\delta_\mathfrak{p}\varepsilon^{-1}_\mathfrak{p}))\rightarrow H^1(H,\Lambda_L(\delta_\mathfrak{p}\varepsilon^{-1}_\mathfrak{p}))$$
is injective. Since $[c_\mathfrak{p}]$ is nontrivial in
$H^1(D_\mathfrak{p},\Lambda_L(\delta_\mathfrak{p}\varepsilon^{-1}_\mathfrak{p}))$,
the homomorphism $\lambda:H\rightarrow\Lambda_L$ is nontrivial. Let
$I$ be the ideal of $\Lambda_L$ generated by $\lambda(H)$. Then $I$
is nonzero and $I$ defines a proper closed subscheme $S$
 of $\mathrm{Spec}(\Lambda_L)$. If $f$ is a classical Hilbert modular form in
 $\mathcal{F}$, then $\rho_f|_H$ is decomposable if and only if $f$
 corresponds an arithmetic point in $S$. Hence for any arithmetic
 point $P$ of $\mathcal{F}$ outside $S$, which corresponds to the
 modular form $f$, the representation $\rho_f|_H$, and hence
 $\rho_f|_{D_\mathfrak{p}}$ is indecomposable.
\end{proof}
Now we consider the nearly ordinary case. Let
$$\mathcal{O}_{F,p}=\lim_\leftarrow \mathcal{O}_F/p^n\mathcal{O}_F$$
be the $p$-adic completion of $\mathcal{O}_F$ at $p$, and $U_F$ be
the torsion free part of $\mathcal{O}_{F,p}^\times$. Then set
$\Gamma=Z_1\times U_F$  and let $\Lambda'=\mathcal{O}_K[[ \Gamma]]$
be the continuous group algebra. For any finite character
$\varepsilon:\Gamma\rightarrow \bar{\mathbb{Q}}_p^\times$, we have
another character
$$\Gamma=Z_1\times U_F\rightarrow \bar{\mathbb{Q}}_p^\times, (a,d)\mapsto \chi(a)^\mu d^v
\varepsilon((a,d)),$$ which induces a homomorphism
$\kappa_{n,v,\varepsilon}:\Lambda'\rightarrow \bar{\mathbb{Q}}_p$.

We briefly recall the definition of nearly ordinary Hecke algebras
defined in \cite{H89a} Section $1$. For any $\alpha\geq 1$, set
$U_\alpha=U_1(\mathfrak{n})\cap U(p^\alpha)$, and let
$\mathbbm{h}_{k,w}(\mathfrak{n}p^\alpha;\mathcal{O}_\Phi)$ be the
$\mathcal{O}_\Phi$-subalgebra of
$\mathrm{End}_\mathbb{C}(S_{k,w}(\mathfrak{n}p^\alpha;\mathbb{C}))$
generated by all the Hecke operators $(U_\alpha x U_\alpha)$ for
$x\in U_0(\mathfrak{n}p^\alpha)$ over $\mathcal{O}_\Phi$. Set
$\mathbbm{h}_{k,w}(\mathfrak{n}p^\alpha;\mathcal{O}_K)=\mathbbm{h}_{k,w}(\mathfrak{n}p^\alpha;\mathcal{O}_\Phi)\otimes_{\mathcal{O}_\Phi}\mathcal{O}_K$.
Applying the ordinary projector $e_\alpha$ we get the nearly
ordinary Hecke algebra
$\mathbbm{h}_{k,w}^{n.ord}(\mathfrak{n}p^\alpha;\mathcal{O}_K)$, and
by taking limit, we have the Hecke algebra
$\mathbbm{h}_{k,w}^{n.ord}(\mathfrak{n}p^\infty;\mathcal{O}_K)$.
From \cite{H89a} Theorem $2.3$, the Hecke algebra
$\mathbbm{h}_{k,w}^{n.ord}(\mathfrak{n}p^\infty;\mathcal{O}_K)$ are
all isomorphic to each other for all pair $(k,w)$ as
$\mathcal{O}_K$-algebras and denote this algebra by
$\mathbbm{h}^{n.ord}(\mathfrak{n}p^\infty;\mathcal{O}_K)$, which is
a torsion free $\Lambda'$-module of finite type. Let
$\mathrm{Spec}(\Lambda'_L)$ be an irreducible component of
$\mathrm{Spec}(\mathbbm{h}^{n.ord}(\mathfrak{n}p^\infty;\mathcal{O}_K))$
and let
$\mathcal{F}:\mathbbm{h}^{n.ord}(\mathfrak{n}p^\infty;\mathcal{O}_K)\rightarrow
\Lambda'_L$ be the corresponding homomorphism. We know that
$\Lambda'_L$ is free of finite rank over $\Lambda'$. A
$\bar{\mathbb{Q}}_p$-rational point $P\in
\mathrm{Spec}(\Lambda'_L)(\bar{\mathbb{Q}}_p)$ is called an
arithmetic point if the corresponding homomorphism $\varphi_P$
extends $\kappa_{n,v,\varepsilon}$ for some $n,v$. For such an
arithmetic point, the composition $\varphi_P\circ(\mathcal{F})$
gives the eigenvalues of a Hilbert modular form of weight $(k,w)$
and tame level $\mathfrak{m}$.

For such an $\mathcal{F}$, we have a two dimensional Galois
representation $\rho_\mathcal{F}:
\Gal(\bar{\mathbb{Q}}/\mathbb{Q})\rightarrow GL_2(\Lambda'_L)$ such
that for any prime $\mathfrak{p}$ of $F$ over $p$, the restriction
$\rho_\mathcal{F}|_{D_\mathfrak{p}}$ is upper triangular. Similarly
with Theorem \ref{Th5}, we have the following result:

\begin{theorem}
Suppose that  $\mathcal{F}$ does not have complex multiplication,
and $\mathcal{F}$ has an arithmetic point $P$ which corresponds to a
(parallel) weight $2$ Hilbert modular form satisfying the condition
required in Theorem \ref{Th4}. Then there exists a proper closed
subscheme $S$ of $\mathrm{Spec}(\Lambda'_L)$ such that for all
arithmetic points $P$ of $\mathrm{Spec}(\Lambda'_L)$ outside $S$
which corresponds to a classical form $f$, the representation
$\rho_f|_{D_{\mathfrak{p}}}$ is indecomposable, where $\rho_f$ is
the Galois representation associated to $f$.

\end{theorem}

\subsection{Application to a problem of Coleman}
In the rest of this paper, we work with elliptic modular forms. Let
$p>3$ be a prime number and $N$ be a positive number prime to $p$.
For each integer $k$ we use $M^\dag_k(\Gamma_1(N))$ (resp.
$S^\dag_k(\Gamma_1(N))$) to denote the space of overconvergent
$p$-adic modular forms (resp. cuspforms) of level $N$ over
$\mathbb{C}_p$ (see \cite{Ka73} for the definitions). In \cite{Co96}
Proposition $6.3$, Coleman proved that there is a linear map
$\theta^{k-1}: M^\dag_{2-k}(\Gamma_1(N))\rightarrow
M^\dag_k(\Gamma_1(N))$ such that the effect of $\theta^{k-1}$ on the
$q$-expansions  is given by the differential operator
$(q\frac{d}{dq})^{k-1}$. Also there is an operator $U$ on
$M^\dag_k(\Gamma_1(N))$ such that if $F(q)=\Sigma_{n\geq 0}a_n q^n$
is an overconvergent modular form, then $U(F)(q)=\Sigma_{n\geq
0}a_{pn}q^n$. Recall that if $F$ is a generalized eigenvector for
$U$ with eigenvalue $\lambda$ in the sense that there exists some
$n\geq 1$ such that $(U-\lambda)^n(F)=0$, then the $p$-adic
valuation of $\lambda$ is called the slope of $F$. From \cite{Co96}
Lemma $6.3$, if $f\in S^\dag_k(\Gamma_1(N))$ is a normalized
classical eigenform of slope strictly smaller than $k-1$, then $f$
cannot be in the image of $\theta^{k-1}$. On the other hand,  a
classical eigenform cannot have slope larger than $k-1$. Then it
remains to consider the remaining boundary case; i.e. overconvergent
modular forms of slope one less than the weight. In \cite{Co96}
Proposition $7.1$, Coleman proved that for $k\geq 2$, every
classical CM cuspidal eigenform of weight $k$ and slope $k-1$ is in
the image of $\theta^{k-1}$. Then he asked whether there is non-CM
classical cusp forms in the image of $\theta^{k-1}$. Since the only
possible slope for new forms of weight $k$ is $\frac{k}{2}-1$ (see
\cite{GM92} Section $4$), it's enough to consider old forms.

Let $g=\Sigma_{n\geq 1}a_n q^n$ be a classical normalized eigenform
of level $N$ and weight $k\geq 2$. Denote by
$K_g=\mathbb{Q}(a_n|n=1,2,\ldots)$ the Hecke field of $g$, which is
known to be a number field. For each prime $\mathfrak{p}$ of $K_g$
over the rational prime $p$, it induces an embedding
$i_\mathfrak{p}:K_g\rightarrow\bar{\mathbb{Q}}_p$ and let
$v_\mathfrak{p}$ be the corresponding valuation on $K_g$. Then we
can regard $g$ as a modular form over $\bar{\mathbb{Q}}_p$ by
$i_\mathfrak{p}$. As explained in \cite{GM92} Section $4$, one can
attach to $g$ two oldforms on $\Gamma_1(N)\cap \Gamma_0(p)$ whose
slopes add up to $k-1$. When the eigenform $g$ is
$\mathfrak{p}$-ordinary; i.e. $v_\mathfrak{p}(a_p)=1$, one of the
associated oldforms has slope $0$ and the other has slope $k-1$ . We
denote the latter oldform by $f$. What we can prove is the
following:

\begin{proposition}
Let $g$ be a weight two normalized classical cusp eigenform on
$\Gamma_1(N)$ with the Hecke field $K_g$. Suppose that there exists
a prime $\mathfrak{p}$ of $K_g$ over the rational prime $p$ such
that $g$ is $\mathfrak{p}$-ordinary, and the associated slope one
oldform $f$ is in the image of the operator $\theta$. Then $g$ is a
CM eigenform.
\end{proposition}

\begin{proof}
Let $\rho_{g,\mathfrak{p}}:
\Gal(\bar{\mathbb{Q}}/\mathbb{Q})\rightarrow
GL_2(K_{g,\mathfrak{p}})$ be the $\mathfrak{p}$-adic Galois
representation attached to $g$. As explained in \cite{Em}
Proposition $1.2$ or \cite{Gh06} Proposition $11$, when $f$ is in
the image of $\theta$, the restriction of $\rho_{g,\mathfrak{p}}$ to
an inertia group $I_p$ of $\Gal(\bar{\mathbb{Q}}/\mathbb{Q})$ at $p$
splits as the direct sum of the trivial character and the character
$\chi_p$, where $\chi_p$ is the $p$-adic cyclotomic character. Then
from Theorem \ref{Th4}, the eigenform $g$ must have complex
multiplication.
\end{proof}

\begin{remark}
In \cite{Em} Theorem $1.3$, Emerton proved that if the assumption in
the above proposition is true for all primes $\mathfrak{p}$ of $K_g$
over $p$, then $g$ is a CM eigenform. Hence the above proposition
can be regarded as an improvement of his theorem. Also in
\cite{Gh06} Section $6$, Ghate discussed the case when $p$ divides
the level $N$. In this case he explained that one can also attach to
the eigenform $g$ a primitive form $f$ with the same weight and
level as $g$. Then he proved that $f$ is in the image of $\theta$ if
and only if the restriction of $\rho_{g,\mathfrak{p}}$ to the
inertia group $I_p$ splits (we need to emphasize here that Ghate's
argument works for all weights, but we restrict ourselves to the
weight two case where Theorem \ref{Th4} is applicable). Hence the
result in
 Theorem \ref{Th4} also applies and the above proposition still
holds in this case.

\end{remark}

%[JH]

%{\tt ToDo-Reminder for me: Quasi-split, split over tame extension - Frobenius never does anything wrong.

%$A_h$ umbenennen. $A_\psi$ oder $A$, vieleicht ist die abh. von $\psi$ f\"ur die funktorialit\"at n\"utzlich.

%Big-cell in das Cor.  8.3.

%$Kl_{\sideset{^L}{}\G}$ - neue notation.

%Wom\"oglich monodromie f\"ur $Kl(\chi)$}

% \begin{enumerate}
% \item Local representation of Gross-Reeder
% \item Gross' use of the trace formula
% \item Frenkel-Gross construction of the connection
% \item Relation with Katz' Kloosterman sheaves
% \item summarize our construction
% \end{enumerate}

Department of Mathematics, UCLA, CA $90095$-$1555$, U.S.A.

E-mail address: zhaobin@math.ucla.edu

\end{document}